\documentclass[11pt,a4paper]{amsart}
\usepackage{layout}
\usepackage{anysize}
\marginsize{22mm}{18mm}{17mm}{18mm}
\usepackage{todonotes}
\usepackage{amsmath,amssymb,amsthm,amscd,amsfonts,comment}
\usepackage[T1]{fontenc}
\usepackage[all]{xy}
\usepackage{hyperref}
\usepackage{graphics,enumitem}
\usepackage{xcolor}
\usepackage{todonotes,comment}

\swapnumbers

\numberwithin{equation}{section}

\theoremstyle{plain}
\newtheorem{theorem}{Theorem}[section]
\newtheorem{proposition}[theorem]{Proposition}
\newtheorem{lemma}[theorem]{Lemma}
\newtheorem{corollary}[theorem]{Corollary}

\theoremstyle{definition}
\newtheorem{definition}[theorem]{Definition}

\theoremstyle{remark}
\newtheorem{remark}[theorem]{Remark}

\let\Im\relax
\DeclareMathOperator{\Im}{Im}
\let\Re\relax
\DeclareMathOperator{\Re}{Re}

\DeclareMathOperator{\PSL}{PSL}

\DeclareMathOperator{\Stab}{Stab}

\DeclareMathOperator{\id}{id}
\DeclareMathOperator{\vol}{vol}

\DeclareMathOperator{\K}{K}

\DeclareMathOperator{\Landau}{O}
\let\mod\relax
\DeclareMathOperator{\mod}{\,mod\,}
\DeclareMathOperator{\Res}{Res}

\DeclareMathOperator{\res}{Res}

\DeclareMathOperator{\hyp}{hyp}

\providecommand{\abs}[1]{\lvert#1\rvert}

\begin{document}

\title{A Kronecker limit type formula for elliptic Eisenstein series}
\author[von Pippich]{Anna-Maria von Pippich}
\date{}
\address{Fachbereich Mathematik -- Technische Universit\"at Darmstadt, Schlo{\ss}gartenstr.7, D-64289 Darmstadt, Germany}
\email{pippich@mathematik.tu-darmstadt.de}

\setcounter{tocdepth}{1}
\setcounter{section}{0}
\maketitle

\begin{abstract}
\noindent
Let $\Gamma\subset\mathrm{PSL}_{2}(\mathbb{R})$ be a Fuchsian subgroup of 
the first kind acting by fractional linear transformations on the upper 
half-plane $\mathbb{H}$, and let $M=\Gamma\backslash\mathbb{H}$ be the 
associated finite volume hyperbolic Riemann surface. Associated to any 
cusp of $M$, there is the classically studied 
non-holomorphic (parabolic) Eisenstein series. In \cite{Kudla:1979qf},
Kudla and Millson studied non-holomorphic (hyperbolic) Eisenstein 
series associated to any closed geodesic on $\Gamma\backslash\mathbb{H}$.
Finally, in \cite{jorgKra2011} and \cite{Pippich06-Ref}, 
so-called elliptic Eisenstein series associated to any elliptic fixed point 
of $M$ were introduced.
In the present article, we prove the meromorphic continuation of 
the elliptic Eisenstein series and we explicitly compute its poles and residues.
Further, we derive a Kronecker limit type formula for elliptic Eisenstein series
for general $\Gamma$.
Finally, for the full 
modular group $\mathrm{PSL}_{2}(\mathbb{Z})$, we give an explicit formula
for the Kronecker's limit functions in terms of holomorphic modular forms.
\end{abstract}

\section{Introduction}
\subsection{Classical Eisenstein series}
Classically, in the theory of holomorphic modular forms, the Eisenstein
series of weight $2k$ ($k\in\mathbb{N}$, $k\geq 2$) for the full modular group $\mathrm{PSL}_
{2}(\mathbb{Z})$ are defined by
\begin{align*}
E_{2k}(z):=\frac{1}{2}\sum_{\substack{(c,d)\in\mathbb{Z}^{2}\\(c,d)=1}}\frac{1}
{(cz+d)^{2k}}\qquad(z\in\mathbb{H}).
\end{align*}
The arithmetic significance of these series is reflected by the fact that their Fourier
coefficients are given by certain divisor sums. More generally, in the theory of automorphic functions,
the non-holomorphic Eisenstein series associated to the cusp $\infty$ of $\mathrm{PSL}_
{2}(\mathbb{Z})$ is defined by
\begin{align*}
\mathcal{E}^{\mathrm{par}}_{\infty}(z,s)
:=\frac{1}{2}\sum_{\substack{\left(c,d\right)\in\mathbb{Z}^{2}\\
\left(c,d\right)=1}}\frac{\Im(z)^s}{\abs{cz+d}^{2s}}
=\sum_{\gamma\in\Gamma_{\infty}\backslash\Gamma}\Im(\gamma z)^{s}
 \qquad
(z\in\mathbb{H}; s\in\mathbb{C},\,\Re(s)>1);
\end{align*}
here $\Gamma_{\infty}$ denotes the stabilizer of the cusp $\infty$ in $\Gamma$.

More generally, let $\Gamma\subset \mathrm{PSL}_{2}(\mathbb{R})$ be a Fuchsian subgroup of the first kind,
then, a non-holomorphic Eisenstein series $\mathcal{E}^{\mathrm{par}}_{p}(z,s)$ associated to any cusp $p$ of 
$\Gamma$ can be defined for $z\in M$ and $s \in \mathbb{C}$ with $\mathrm{Re}(s)>1$. As a function of $s$, this so-called 
parabolic Eisenstein series is a holomorphic function for $\mathrm{Re}(s)>1$ and admits
a meromorphic continuation to the whole comples $s$-plane. The parabolic Eisenstein series play an important 
role in number theory, in particular in the spectral theory of automorphic functions on $M$.

\subsection{Hyperbolic and elliptic Eisenstein series}

In \cite{Kudla:1979qf}, Kudla and Millson first studied a generalized, form-valued Eisenstein series associated to 
any closed geodesic of $M$, or equivalently to any primitive, hyperbolic element of $\Gamma$. A scalar-valued, 
hyperbolic Eisenstein series was defined in \cite{Pippich03-Ref}, and the authors proved that the series admits a 
meromorphic continuation to all $s\in\mathbb{C}$. \\
Analogously, associated to any elliptic fixed point $w$ of $M$, and, more generally, associated to any point 
of $M$, there is an elliptic Eisenstein series, which was first introduced in the unpublished document
\cite{Jorgenson:2004el} by Jorgenson and Kramer. The motivation to consider these series has arisen from 
Arakelov geometry. To compare the canonical (or Arakelov) metric $\mu_{\mathrm{can}}$ and 
the scaled hyperbolic metric $\mu_{\mathrm{shyp}}$, one considers the quotient
\begin{align*}
d_{\Gamma}:=\sup_{z\in\Gamma\backslash\mathbb{H}}
\frac{\mu_{\mathrm{can}}(z)}{\mu_{\mathrm{shyp}}(z)}.
\end{align*}
The main result proven in
\cite{jorgKra2011} is the bound $d_{\Gamma}=\Landau_{\Gamma_0}(1)$, where
$\Gamma\backslash\mathbb{H}$ is a finite degree cover of $\Gamma_0\backslash\mathbb{H}$.
The proof in \cite{jorgKra2011} uses an identity  which relates $\mu_{\mathrm{can}}$ 
to $\mu_{\mathrm{shyp}}$ and an integral involving the hyperbolic heat kernel on 
$\Gamma\backslash\mathbb{H}$ (see \cite{Jorgenson:2006zu}). Decomposing
this heat kernel into terms involving parabolic, hyperbolic, and elliptic elements
of $\Gamma$ (a method that is also used in the proof of Selberg's trace formula),
each term can be bounded by the special value of a parabolic, hyperbolic, and elliptic
Eisenstein series at $s=2$, respectively.

For $z,w\in M$ with $z\not=w$, the elliptic Eisenstein series associated to the point $w\in M$
for the subgroup $\Gamma$ is defined by
\begin{align*}
\mathcal{E}^{\mathrm{ell}}_{w}(z,s)=\sum_{\gamma\in\Gamma_{w}\backslash\Gamma}
\sinh\bigl(d_{\mathrm{hyp}}(e_{j},\gamma z)\bigr)^{-s};
\end{align*}
here, $\Gamma_{w}$ denotes the stabilizer of $w$ in $\Gamma$,
and $d_{\mathrm{hyp}}(w,\gamma z)$ denotes the hyperbolic distance from $w$ to $\gamma z$.
The parabolic, hyperbolic, and elliptic Eisenstein series fulfill various relations. For example, if one 
considers a sequence of elliptically degenerating hyperbolic Riemann surfaces, then the elliptic 
Eisenstein series associated to the degenerating 
elliptic point converges to the rescaled parabolic Eisenstein series associated to the newly developed cusp of 
the limit surface \cite{Garbin:2009uq}. An analogous result was proven in \cite{Garbin:2008ve} for the hyperbolic Eisenstein series.
In \cite{Pippich06-Ref}, the Fourier expansion 
of the elliptic Eisenstein series for the full modular group $\mathrm{PSL}_{2}(\mathbb{Z})$ 
was computed and its meromorphic continuation was established from this.

In the present article, the meromorphic continuation of the elliptic Eisenstein series $\mathcal{E}^{\mathrm{ell}}_{w}(z,s)$ 
for any Fuchsian subgroup $\Gamma$ of the first kind to the whole $s$-plane is proven using methods from the spectral theory
on $M$.
More precisely, we employ a relation between $\mathcal{E}^{\mathrm{ell}}_{w}(z,s)$ 
and an elliptic Poincar\'e series, which is square-integrable on $M$ and which can be meromorphically continued via its spectral expansion. 
Furthermore, we determine the possible poles of $\mathcal{E}^{\mathrm{ell}}_{w}(z,s)$
and we compute its residues. 

\subsection{Kronecker limit type formula}

Classically, for $\Gamma=\mathrm{PSL}_2(\mathbb{Z})$, 
Kronecker's limit formula evaluates the special value of the parabolic Eisenstein series at $s=0$ in terms of the 
logarithm of the absolute value of Dedekind's Delta function. More precisely, at $s=0$, there is a Laurent expansion
of the form
\begin{align}
\mathcal{E}^{\mathrm{par}}_{\infty}(z,s)=
1+\log\bigl(|\Delta(z)|^{1/6}\Im(z)\bigr)
\cdot s+\Landau(s^2)\label{klf1},
\end{align}
with Dedekind's Delta function given by
\begin{align*}
\Delta(z)=\frac{1}{1728}\bigl(E_4(z)^3-E_6(z)^2\bigr),
\end{align*}
For general Fuchsian subgroups of the first kind, analogues of Kronecker's limit formula
were investigated in \cite{Gold:1973rs}. \\
The special value of the hyperbolic Eisenstein series 
at $s=0$ is a harmonic form which is the Poincar\'e dual to the considered geodesic \cite{Kudla:1979qf}.

In the present paper, we first prove a Kronecker limit type formula for elliptic Eisenstein series
for an arbitrary Fuchsian subgroup 
$\Gamma\subset \mathrm{PSL}_{2}(\mathbb{R})$ of the first kind. Having substracted an expression involving
all parabolic Eisenstein series for $\Gamma$, this formula expresses the special value of 
$\mathcal{E}^{\mathrm{ell}}_{w}(z,s)$ at $s=0$ in terms of the norm of a holomorphic modular 
form which vanishes only at the elliptic point $w$.
In case that $\Gamma=\mathrm{PSL}_2(\mathbb{Z})$, we then explicitly determine this holomorphic modular 
form, and we prove that, at $s=0$, there are Laurent expansions of the form
\begin{align}
\mathcal{E}^{\mathrm{ell}}_{i}(z,s)
&=-\log\bigl(|E_{6}(z)|\,|\Delta(z)|^{-1/2}\bigr)\cdot s+\Landau(s^2)\label{klf2},\\
\mathcal{E}^{\mathrm{ell}}_{\rho}(z,s)
&=-\log\bigl(|E_{4}(z)|\,|\Delta(z)|^{-1/3}\bigr)\cdot s+\Landau(s^2)\label{klf3}
\end{align}
with $\rho:=\exp(2\pi/3)$.
Combining the Kronecker's limit formulae \eqref{klf1}, \eqref{klf2}, and \eqref{klf3}, one easily concludes
for a modular form $f$ for $\mathrm{PSL}_2(\mathbb{Z})$, that $\log|f|$ can be expressed as
the special value of a combination of the Eisenstein series $\mathcal{E}^{\mathrm{par}}_{\infty}(z,s)$,
 $\mathcal{E}^{\mathrm{ell}}_{i}(z,s)$, and $\mathcal{E}^{\mathrm{ell}}_{\rho}(z,s)$.
Further, we note that the leading terms in \eqref{klf2} resp.~\eqref{klf3} are equal to
$-\log\bigl(|j(i)-j(z)|^2\bigr)$ and $-\log\bigl(|j(i)-j(z)|^3\bigr)$, respectively. \\
This phenomenon is explained at the end of this article, by proving a
relation between the elliptic Eisenstein series and the automorphic Green's function on $M$,
which is a fundamental object in arithmetic geometry, see, e.g., \cite{gr84}. 
In particular, we prove that, for an arbitrary Fuchsian subgroup 
$\Gamma\subset \mathrm{PSL}_{2}(\mathbb{R})$,
 the elliptic Eisenstein series and the rescaled automorphic Green's function on $M$ coincide at $s=0$.

\subsection{Outline of the article.}
The paper is organized as follows. In section 2, we recall and summarize basic notation
and definitions used in this article, and we cite relevant results from the literature.
In section 3, we define the elliptic Eisenstein series $\mathcal{E}^{\mathrm{ell}}_{w}(z,s)$
for a Fuchsian subgroup $\Gamma\subset \mathrm{PSL}_{2}(\mathbb{R})$ of the first kind
and we prove some of its basic properties. In particular, we show that
the series is holomorphic for $\Re(s)>1$ and an automorphic function for $\Gamma$. 
In contrast to the parabolic series, the elliptic Eisenstein series fails to be 
an eigenfunction of $\Delta_{\mathrm{hyp}}$; instead it satisfies a differential difference equation.
In section 4, we prove the meromorphic continuation of the elliptic Eisenstein series to 
the whole $s$-plane and we determine its poles and compute its residues.
In section 5, we prove a Kronecker limit type formula for the elliptic Eisenstein series for an arbitrary 
Fuchsian subgroup $\Gamma\subset \mathrm{PSL}_{2}(\mathbb{R})$ of the first kind. 
In section 6, we establish the explicit Kronecker's limit formulae \eqref{klf2} and \eqref{klf3} for 
the full modular group $\mathrm{PSL}_2(\mathbb{Z})$.
In section 7, we conclude by proving a relation between the elliptic Eisenstein series and the 
automorphic Green's function on $M$ for an arbitrary Fuchsian subgroup 
$\Gamma\subset \mathrm{PSL}_{2}(\mathbb{R})$ of the first kind.

\subsection*{Acknowledgements}
This paper contains unpublished material of the author's dissertation thesis.
The author would like to thank J\"urg Kramer and Jay Jorgenson for
their advice and helpful discussions.  

\section{Background material}\label{section_bgk}
\subsection{Basic notation}
As mentioned in the introduction, we let $\Gamma\subset\mathrm{PSL}_{2}(\mathbb{R})$ denote a Fuchsian
group of the first kind acting by fractional
linear transformations on the hyperbolic upper half-plane $\mathbb{H}:=\{z=x+iy\in\mathbb{C}\,
|\,x,y\in\mathbb{R};\,y>0\}$. We let $M:=\Gamma\backslash\mathbb{H}$, which is a finite
volume hyperbolic Riemann surface, and denote by $p:\mathbb{H}\longrightarrow M$
the natural projection. By $P_{\Gamma}$ and $E_{\Gamma}$ we denote a complete set 
of $\Gamma$-inequivalent cusps and elliptic fixed points of $\Gamma$,
respectively, and we set
$p_{\Gamma}:=\sharp P_{\Gamma}$, $e_{\Gamma}:=\sharp E_{\Gamma}$,
that is, we assume that $M$ has $e_{\Gamma}$
elliptic fixed points and $p_{\Gamma}$ cusps.
We identify $M$ locally with its universal cover $\mathbb{H}$. 

We let $\mu_{\mathrm{hyp}}$ denote the hyperbolic metric on $M$, which is compatible with the
complex structure of $M$, and has constant negative curvature equal to minus one.
The hyperbolic line element $ds^{2}_{\hyp}$, resp.~the hyperbolic Laplacian
$\Delta_{\hyp}$, are given as
\begin{align*}
ds^{2}_{\hyp}:=\frac{dx^{2}+dy^{2}}{y^{2}},\quad\textrm{resp.}
\quad\Delta_{\hyp}:=-y^{2}\left(\frac{\partial^{2}}{\partial
x^{2}}+\frac{\partial^{2}}{\partial y^{2}}\right).
\end{align*}
By $\mathcal{F}_{\Gamma}$ we denote a fundamental domain of $\Gamma$. Since $\Gamma$
is of the first kind, the hyperbolic volume $\vol_{\hyp}(M)=\vol_{\hyp}(\mathcal{F}_{\Gamma})$
is finite.

By $d_{\mathrm{hyp}}(z,w)$ we denote the hyperbolic distance from $z\in\mathbb{H}$ to
$w\in\mathbb{H}$, which satisfies the relation
\begin{align}\label{cosh-hypdist}
\cosh\bigl(d_{\mathrm{hyp}}(z,w)\bigr)= 1+2 u(z,w)
\end{align}
with the point-pair invariant
\begin{align}
\label{def_u}
u(z,w)=\frac{\left|z-w\right|^{2}}{4\,\mathrm{Im}(z)\mathrm{Im}(w)}\,.
\end{align}
For $z=x+iy\in\mathbb{H}$, we define the hyperbolic polar coordinates $\varrho=
\varrho(z),\vartheta=\vartheta(z)$ centered at $i\in\mathbb{H}$ by
\begin{align*}
\varrho(z):=d_{\mathrm{hyp}}(i,z)\,,\quad\vartheta(z):=\measuredangle(\mathcal{L},
T_{z})\,,
\end{align*}
where $\mathcal{L}:=\{z\in\mathbb{H}\,|\,x=\mathrm{Re}(z)=0\}$ denotes the
positive $y$-axis and $T_{z}$ is the euclidean tangent at the unique geodesic
passing through $i$ and $z$ at the point $i$. In terms of the hyperbolic polar
coordinates, the hyperbolic line element, resp.~the hyperbolic Laplacian take
the form
\begin{align*}
ds_{\hyp}^{2}=\sinh^{2}(\varrho)d\vartheta^{2}+d\varrho^{2},
\quad\textrm{resp.}\quad
\Delta_{\hyp}=-\frac{\partial^{2}}{\partial\varrho^{2}}-\frac{1}{\tanh(\varrho)}
\frac{\partial}{\partial\varrho}-\frac{1}{\sinh^{2}(\varrho)}\frac{\partial^{2}}
{\partial\vartheta^{2}}\,.
\end{align*}

\subsection{Parabolic Eisenstein series}\label{subsec_pareis}
For a parabolic fixed point/cusp $p_j\in P_{\Gamma}$ $(j=1,\dots, p_{\Gamma})$, 
let $\Gamma_{p_j}$ denote its stabilizer subgroup
$
 \Gamma_{p_j}:=\Stab_{\Gamma}(p_j)=\langle \gamma_{p_j}\rangle
$ 
generated by a primitive, parabolic element $\gamma_{p_j}\in\Gamma$. 
Further, let $\sigma_{p_j}\in\PSL_{2}(\mathbb{R})$ be a (parabolic) scaling matrix,
that is a matrix satisfying 
\begin{align*}
 \sigma_{p_j}\infty&=p_j ,
\\
 \sigma_{p_j}^{-1} \gamma_{p_j}\sigma_{p_j}&=
\begin{pmatrix}
1&1\\ 0&1
\end{pmatrix} 
;                                      
\end{align*}
the scaling matrix $\sigma_{p_j}$ is unique up to multiplication on the right by 
elements $\bigl(\begin{smallmatrix}1&\xi\\ 0&1\end{smallmatrix}\bigr)$ with
$\xi\in\mathbb{R}$. 

For $z\in\mathbb{H}$ and $s\in\mathbb{C}$, the \textit{parabolic Eisenstein series associated to the 
parabolic fixed point} $p_{j}\in P_{\Gamma}$ $(j=1,\dots,p_{\Gamma})$ is defined by
\begin{align}
\mathcal{E}^{\mathrm{par}}_{p_{j}}(z,s)=\sum_{\gamma \in \Gamma_{p_{j}}\backslash \Gamma}
                         \Im\bigl(\sigma_{p_{j}}^{-1}\gamma z\bigr)^{s}.
\label{def_eis_par}
\end{align}
Referring to \cite{Hejhal:1983vn},  \cite{Iwaniec:2002zr}, or \cite{Kubota:1973uq},
where detailed proofs are provided, we recall that the series \eqref{def_eis_par} converges 
absolutely and locally uniformly for any $z\in\mathbb{H}$ and $s\in\mathbb{C}
$ with $\mathrm{Re}(s)>1$, and that it is holomorphic for $s\in\mathbb{C}$
with $\mathrm{Re}(s)>1$. Moreover, it is invariant with respect to $\Gamma$, 
i.e.~$\mathcal{E}^{\mathrm{par}}_{p_{j}}(z,s)\in\mathcal{A}(\Gamma\backslash\mathbb{H})$.
A straightforward computation shows that 
the series \eqref{def_eis_par} satisfies the differential equation
\begin{align*}
\big(\Delta_{\hyp}-s(1-s)\big)\mathcal{E}^{\mathrm{par}}_{p_{j}}(z,s)=0\,,
\end{align*}
i.e.~it is an eigenfunction of $\Delta_{\hyp}$ and, therefore, it is a real-analytic function
(with respect to $z=x+iy$).

The parabolic Eisenstein series $\mathcal{E}^{\mathrm{par}}_{p_{j}}(z,s)$ $(j=1,\dots,p_{\Gamma})$ 
admits a meromorphic continuation
to the whole $s$-plane. For $\Re(s)\geq1/2$, there are only finitely many poles; they are located 
in the interval $(1/2,1]$ and they are simple. There is always a pole at $s=1$ with residue
\begin{align*}
\res_{s=1}\mathcal{E}^{\mathrm{par}}_{p_{j}}(z,s)=\frac{1}{\vol_{\mathrm{hyp}}(\mathcal{F}_{\Gamma})}\,.
\end{align*}
Moreover, the following functional equations hold 
\begin{align}
\label{func_par_eis}
\mathcal{E}^{\mathrm{par}}_{p_{j}}(z,s)=
\sum_{k=1}^{p_{\Gamma}}
\varphi_{p_j,p_{k}}(s)\,
\mathcal{E}^{\mathrm{par}}_{p_{k}}(z,1-s)
\end{align}
and
\begin{align}
\label{func_scattering}
\sum_{\ell=1}^{p_{\Gamma}}
\varphi_{p_j,p_{\ell}}(s)\,\varphi_{p_{\ell},p_{k}}(1-s)
=\delta_{p_j,p_k};
\end{align}
here
\begin{align*}
\varphi_{p_j,p_{k}}(s):=
\frac{\sqrt{\pi}\,\Gamma\bigl(s-\frac{1}{2}\bigr)}
{\Gamma(s)}\,
 \sum_{c=1}^{\infty} \frac{1}{c^{2s}}
\Bigl(\sum_{\substack{ d\mod c\\ 
\bigl(\begin{smallmatrix} *&*\\ c&d\end{smallmatrix} \bigr)
\in\sigma_{p_{j}}^{-1}\Gamma\sigma_{p_{k}}}}1
\Bigr)
\end{align*}
are the scattering entries (see, e.g., \cite{Iwaniec:2002zr} or \cite{Kubota:1973uq}).

\subsection{Spectral expansions.} 
We enumerate the discrete eigenvalues of the hyperbolic Laplacian
$\Delta_{\mathrm{hyp}}$ acting on smooth functions on $M$ by
$$
0=\lambda_{0}<\lambda_{1}\leq\lambda_{2}\leq\ldots;
$$
we write $\lambda_{r}=1/4+t_{r}^{2}=s_{r}(1-s_{r})$, i.e., $s_{r}=1/2+it_{r}$ 
with $t_{r}>0$ or $t_{r}\in[-i/2,0]$. An eigenvalue $\lambda_{r}$ with $0<\lambda_{r}<1/4$, 
that is $1/2<s_r<1$, is called exceptional.
In case that $p_{\Gamma}>0$, $\Delta_{\hyp}$ has also a continuous spectrum covering the interval 
$[1/4, \infty)$ uniformly with multiplicity $1$. The eigenvalues of the continuous spectrum are of 
the form $\lambda=1/4+t^{2}=s(1-s)$, i.e., $s=1/2+it$ with $t\in\mathbb{R}$. The corresponding 
eigenfunctions are given by the parabolic Eisenstein series $\mathcal{E}^{\mathrm{par}}_{p_k}(z,1/2+it)$ 
($k=1,\dots,p_{\Gamma}$).

Under certain hypotheses on a function $f$ on $M$, which are defined carefully in numerous references such as
\cite{Iwaniec:1997ys}, \cite{Hejhal:1983vn}, or \cite{Kubota:1973uq}, there is a spectral expansion in terms of the eigenfunctions
$\psi_{r}$ associated to the discrete eigenvalues $\lambda_{r}$ of
the hyperbolic Laplacian $\Delta_{\hyp}$ and the parabolic
Eisenstein series $\mathcal{E}^{\mathrm{par}}_{p_k}$ associated to the
cusps $p_k\in P_{\Gamma}$; without loss of generality, we
assume that all eigenfunctions of the Laplacian are real-valued.
\begin{align}
\label{def_spectral_exp}
f(z)=\sum\limits_{r=0}^{\infty}\langle f,\psi_{r}\rangle\,\psi_{r}(z)+
\frac{1}{4\pi}\sum\limits_{k=1}^{p_{\Gamma}}\,\int\limits_{-\infty}^{\infty}
\big\langle
 f,{\textstyle\mathcal{E}^{\mathrm{par}}_{p_{k}}\bigl(\cdot,\frac{1}{2}+it\bigr)}
\big\rangle\,
{\textstyle\mathcal{E}^{\mathrm{par}}_{p_{k}}\bigl(z,\frac{1}{2}+it\bigr)}dt.
\end{align}
By $\mathcal{L}^2(M)$ we denote the space of square-integrable functions on $M$.
For any $f\in\mathcal{L}^2(M)$ such that $f$ and $\Delta_{\mathrm{hyp}}f$ 
are smooth and bounded, the series $\eqref{def_spectral_exp}$ converges 
absolutely and uniformly for $z$ ranging over compacta $K\subseteq \mathbb{H}$.

\subsection{Elliptic Poincar\'e series}\label{subsection_ell_poinc}
For an elliptic fixed point $e_j\in E_{\Gamma}$ $(j=1,\dots, e_{\Gamma})$, 
let $ \Gamma_{e_j}$ denote its stabilizer subgroup
$
 \Gamma_{e_j}:=\Stab_{\Gamma}(e_j)=\langle \gamma_{e_j} \rangle
$ 
generated by a primitive, elliptic element $\gamma_{e_j}\in\Gamma$. 
Hence there is an elliptic scaling matrix $\sigma_{e_j}\in\mathrm{PSL}_{2}(\mathbb{R})$ 
satisfying
\begin{align*}
\sigma_{e_j}i&=e_j,
\end{align*}
i.e., we have
\begin{align*}
\sigma_{e_j}^{-1} \gamma_{e_j}\sigma_{e_j}&=\begin{pmatrix}    
\cos(\pi/n_{e_j})&\sin(\pi/n_{e_j})\\ -\sin(\pi/n_{e_j})& \cos(\pi/n_{e_j})=:k(\pi/n_{e_j})
\end{pmatrix},
\end{align*}
where
 $$
 n_{z}:=\mathrm{ord}_{\Gamma}(z)
 \qquad (n_{z}\in\mathbb{N})
 $$ 
denotes the order of the point $z$ in $\Gamma$; if $z$ is an elliptic fixed point, then $n_z>1$.
Therefore, the group $\sigma_{e_j}^{-1}\Gamma\sigma_{e_j}\subset\PSL_2(\mathbb{R})$ is a Fuchsian subgroup of the first kind with the elliptic fixed point $i$;
the scaling matrix $\sigma_{e_j}$ is unique up to multiplication on the right by elements
$k(\theta)\in\K$, where $\theta\in[0,2\pi)$.

For $z\in\mathbb{H}$, $s\in\mathbb{C}$, the \textit{elliptic Poincar\'e series
$P^{\mathrm{ell}}_{e_j}$ associated to the elliptic fixed point} 
$e_j\in E_{\Gamma}$ 
is defined by
\begin{align}
P^{\mathrm{ell}}_{e_j}(z,s)
&=\sum_{\gamma\in\Gamma_{e_j}\backslash\Gamma}
\cosh\bigl(\varrho(\sigma_{e_{j}}^{-1}\gamma z)\bigr)^{-s}
\label{eq_ellpoin_P}
\end{align}
For fixed $z\in\mathbb{H}$, the elliptic Poincar\'e series $P^{\mathrm{ell}}_{e_j}(z,s)$ 
converges absolutely and locally uniformly for $s\in\mathbb{C}$ with $\Re(s)>1$, 
and hence defines a holomorphic function. The elliptic Poincar\'e series 
$P^{\mathrm{ell}}_{e_j}(z,s)$ is invariant under the
action of $\Gamma$, i.e. we have 
$$P^{\mathrm{ell}}_{e_j}(\gamma z,s)=P^{\mathrm{ell}}_{e_j}(z,s)$$
for any $\gamma\in\Gamma$. For fixed $s\in\mathbb{C}$ with $\Re(s)>1$, the elliptic Poincar\'e series 
$P^{\mathrm{ell}}_{e_j}(z,s)$ converges absolutely and uniformly for 
$z$ ranging over compacta $K\subseteq\mathbb{H}$. Moreover, we have
$P^{\mathrm{ell}}_{e_j}(z,s)\in\mathcal{L}^2(M)$.
For $z\in\mathbb{H}$ and $s\in\mathbb{C}$ with $\Re(s)>1$, the series 
$P^{\mathrm{ell}}_{e_j}(z,s)$ satisfies the differential equation
$(j=1,\dots, e_{\Gamma})$ 
\begin{align*}
\bigl(\Delta_{\hyp}-s(1-s)\bigr)P^{\mathrm{ell}}_{e_j}(z,s)=
s(s+1)P^{\mathrm{ell}}_{e_j}(z,s+2).
\end{align*}

\subsection{Special functions}\label{subsection_specfct}
For special functions, we refer to the vast literature (see, e.g., \cite{Gradshteyn:2007ys} or \cite{Erdelyi:1981uq}). 
However, for the convenience of the reader, we recall several useful identities.

The $\Gamma$-function satisfies the following duplication formula
\begin{align}
\label{eq-duplication-Gamma}
\Gamma(s)\Gamma\Bigl(s+\frac{1}{2}\Bigr)=\sqrt{\pi}\,2^{1-2s}\Gamma(2s).
\end{align}
By Stirling's asymptotic formula, we have for fixed $\sigma\in\mathbb{R}$ and $|t|\rightarrow
\infty$ the asymptotics
\begin{align*}
\big|\Gamma(\sigma+it)\big|\sim\sqrt{2\pi}|t|^{\sigma-1/2}
\exp\Bigl(-\frac{\pi |t|}{2}\Bigr)
\end{align*}
with an implied constant depending on $\sigma$. 

For $m\in\mathbb{N}$, the Pochhammer symbol is given by
$$
(s)_{m}:=\frac{\Gamma(s+m)}{\Gamma(s)};
$$ 
for $m\in\mathbb{N}$ with $m>0$, we note the 
alternative formula $(s)_{m}=\prod_{j=0}^{m-1}(s+j)$. 
Furthermore, for $m\in\mathbb{Z}$, we have
\begin{align}
\label{eq_gamma_minusn}
\Gamma(s-m)&=\frac{(-1)^{m}\,\Gamma(s)}{(1-s)_{m}}.
\end{align}
There are various addition formulas; we will made use of the formula
\begin{align}
(a+b)_{n}&=\sum\limits_{k=0}^{n}
{n \choose k}
(a)_{k}(b)_{n-k}
\label{sum_pochh2}
\end{align}
which holds for $n\in\mathbb{N}$.

Finally, for $a,b,c\in\mathbb{C}$, $c\neq -n$ ($n\in\mathbb{N}$), 
and $w\in\mathbb{C}$, we denote Gauss's hypergeometric function by $F(a,b;c;w)$; it is defined by the series
\begin{align*}
F(a,b;c;w):=\sum_{k=0}^{\infty}\frac{(a)_{k}(b)_{k}}{(c)_{k}\, k!}\, w^{k},
\end{align*}
which converges absolutely for $w\in\mathbb{C}$ with $|w|<1$.

\section{Definition and basic properties}
In this section, we define the elliptic Eisenstein series and we prove some of its basic properties.

\begin{definition}\label{4.1}
For $z\in\mathbb{H}$ with $z\neq\gamma e_{j}$ for any $\gamma\in\Gamma$,
and $s\in\mathbb{C}$, the \textit{elliptic Eisenstein series associated to the 
elliptic fixed point} $e_j\in E_{\Gamma}$ $(j=1,\dots,e_{\Gamma})$ is defined by
\begin{align*}
\mathcal{E}^{\mathrm{ell}}_{e_{j}}(z,s)=\sum_{\gamma\in\Gamma_{e_{j}}\backslash\Gamma}
\sinh\bigl(\varrho(\sigma_{e_{j}}^{-1}\gamma z)\bigr)^{-s};
\end{align*}
here, are $\sigma_{e_{j}}$ and $\Gamma_{e_{j}}$ are defined in subsection \ref{subsection_ell_poinc}.
\end{definition}

\begin{lemma}\label{4.2}
The following assertions hold
\begin{itemize}
\item[(i)]
For fixed $z\in\mathbb{H}$ with $z\neq\gamma e_{j}$ for any $\gamma\in\Gamma$,
the elliptic Eisenstein series $\mathcal{E}^{\mathrm{ell}}_{e_{j}}(z,s)$ converges absolutely and
locally uniformly for $s\in\mathbb{C}$ with $\Re(s)>1$, and hence defines a holomorphic
function.
\item[(ii)] 
The elliptic Eisenstein series $\mathcal{E}^{\mathrm{ell}}_{e_{j}}(z,s)$ is invariant under the
action of $\Gamma$, i.e. we have 
$$\mathcal{E}^{\mathrm{ell}}_{e_{j}}(\gamma z,s)=\mathcal{E}^{\mathrm{ell}}_{e_{j}}(z,s)$$
for any $\gamma\in\Gamma$, and fulfills $n_{e_j}\mathcal{E}^{\mathrm{ell}}_{e_{j}}(z,s)=n_z\mathcal{E}^{\mathrm{ell}}_{z}(e_j,s)$.
\item[(iii)]
For fixed $s\in\mathbb{C}$ with $\Re(s)>1$, the elliptic Eisenstein series 
$\mathcal{E}^{\mathrm{ell}}_{e_{j}}(z,s)$ converges absolutely and uniformly for 
$z$ ranging over compacta $K\subseteq\mathbb
{H}$ not containing any translate $\gamma e_j$ of $e_j$ by $\gamma\in\Gamma$.
\end{itemize}
\end{lemma}
\begin{proof}
(i) We first assume that $e_{1}=i$ is an elliptic fixed point
of $\Gamma$, and we will prove the 
assertion for the elliptic Eisenstein series 
\begin{align*}
\mathcal{E}^{\mathrm{ell}}_{i}(z,s)
=
\sum\limits_{\gamma \in \Gamma_{i}\backslash \Gamma} \sinh\bigl(\varrho(\gamma z)\bigr)^{-s}  
=              
\frac{1}{\mathrm{ord}(i)}\,\sum\limits_{\gamma \in \Gamma} \sinh\bigl(\varrho(\gamma z)\bigr)^{-s}.
\end{align*} 
To do this, we write $s=\sigma+it\in\mathbb{C}$ and we assume that $\sigma=\Re(s)>1$.
We fix $z\in\mathbb{H}$ such that $z\neq\gamma i$ for any $\gamma\in\Gamma$. Since $\Gamma$
acts properly discontinously on $\mathbb{H}$ and $z\neq\gamma i$ for any $\gamma\in\Gamma$,
the minimum
\begin{align*}
R_{1}(z):=\min_{\gamma\in\Gamma}d_{\hyp}(i,\gamma z)
=\min_{\gamma\in\Gamma}\varrho(\gamma z)
\end{align*}
exists and is strictly positive. Hence, introducing the quantity
\begin{align*}
C_{1}(z):=\frac{1-\exp\bigl(-2R_{1}(z)\bigr)}{2}>0,
\end{align*}
we derive the estimate
\begin{align*}
\sinh\bigl(\varrho(\gamma z)\bigr)
=\exp\bigl(\varrho(\gamma z)\bigr)\,\frac{1-\exp\bigl(-2\varrho(\gamma z)\bigr)}{2}
\geq C_{1}(z)\,\exp\bigl(\varrho(\gamma z)\bigr)
\end{align*}
for all $\gamma\in\Gamma$. Therefore, we obtain the estimate
\begin{align}
\sum\limits_{\gamma\in\Gamma_{i}\backslash\Gamma}\Bigl|\sinh\bigl(\varrho(\gamma z)\bigr)^{-s}\Bigr|
=\frac{1}{\mathrm{ord}(i)}\,\sum\limits_{\gamma\in\Gamma}\sinh\bigl(\varrho(\gamma z)\bigr)^{-\sigma}
\leq
\frac{C_{1}(z)^{-\sigma}}{\mathrm{ord}(i)}\,\sum\limits_{\gamma\in\Gamma}\exp\bigl(-\sigma\varrho
(\gamma z)\bigr).
\label{equ_poin_series}
\end{align}
Since the Poincar\'e exponent of $\Gamma$ equals 1, the series on the right-hand side
of \eqref{equ_poin_series} converges locally uniformly for $\sigma>1$. This complete the proof
in the case that $e_j=i$.\\
Now, let $e_j$ be an arbitrary elliptic fixed point of $\Gamma$. 
Then, $\sigma_{e_{j}}^{-1}\Gamma\sigma_{e_{j}}$ has the elliptic fixed point $i$
and, for $\Re(s)>1$ and $w=\sigma_{e_{j}}^{-1}z$, we have the equality
\begin{align*}
\mathcal{E}^{\mathrm{ell}}_{\Gamma,e_{j}}(z,s)=
\mathcal{E}^{\mathrm{ell}}_{\Gamma,e_{j}}( \sigma_{e_{j}} w,s)=
\mathcal{E}^{\mathrm{ell}}_{\sigma_{e_{j}}^{-1}\Gamma\sigma_{e_{j}},i}(w,s)
\end{align*}
Furthermore, $w$ is equivalent to $i$ with respect to $\sigma_{e_{j}}^{-1}\Gamma\sigma_{e_{j}}$ 
if and only if $z=\sigma_{e_{j}} w$ is equivalent to $e_j$ with respect to $\Gamma$.
The abolute and locally uniform convergence of
$\mathcal{E}^{\mathrm{ell}}_{\sigma_{e_{j}}^{-1}\Gamma\sigma_{e_{j}},i}(w,s)$ 
($w$ not equivalent to $i$ with respect to $\sigma_{e_{j}}^{-1}\Gamma\sigma_{e_{j}}$) for $\Re(s)>1$ 
for fixed $w\in\mathbb{H}$ therefore implies the absolute uniform convergence 
of $\mathcal{E}^{\mathrm{ell}}_{\Gamma,e_{j}}(z,s)$ ($z$ not equivalent to $e_j$ with respect to $\Gamma$)
for $\Re(s)>1$ for fixed $z=\sigma_{e_{j}}w\in\mathbb{H}$.
This completes the proof of assertion (i).

(ii) From Definition \ref{4.1} we immediately deduce for $s\in\mathbb{C}$ with $\Re(s)>1$
\begin{align*}
\mathcal{E}^{\mathrm{ell}}_{e_{j}}(\gamma z,s)=\mathcal{E}^{\mathrm{ell}}_{e_{j}}(z,s)
\end{align*}
for all $\gamma\in\Gamma$ and $n_{e_j}\mathcal{E}^{\mathrm{ell}}_{e_{j}}(z,s)=n_z\mathcal{E}^{\mathrm{ell}}_{z}(e_j,s)$, 
provided that $z\neq\gamma e_{j}$ for any $\gamma\in\Gamma$.

(iii) Let finally $K\subseteq\mathbb{H}$ be a compact subset not containing any translate
$\gamma e_{j}$ of $e_{j}$ by $\gamma\in\Gamma$. Then, the expression
$C_{1}(z)\cdot\exp\bigl(\varrho(\gamma z)\bigr)$
given in the first part of the proof can be bounded uniformly for all $z\in K$. For
fixed $s\in\mathbb{C}$ with $\Re(s)>1$, the series $\mathcal{E}^{\mathrm{ell}}_{e_{j}}(z,s)$ 
therefore converges absolutely and uniformly on $K\subseteq\mathbb{H}$.
\end{proof} 

\begin{lemma}\label{lemm_elleis_3}
For $z\in\mathbb{H}$ with $z\neq\gamma e_{j}$ for any $\gamma\in\Gamma$,
and $s\in\mathbb{C}$ with $\Re(s)>1$, the elliptic Eisenstein series $\mathcal{E}^{\mathrm{ell}}_{e_{j}}(z,s)$ 
can be written as
\begin{align*}
\mathcal{E}^{\mathrm{ell}}_{e_{j}}(z,s)
&=\sin\Bigl(\frac{\pi}{n_{e_{j}}}\Bigr)^{s}
 \sum_{\gamma \in \mathcal{K}_{e_{j}}}
  u(z,\gamma z)^{-s/2}
\end{align*}
with the $\Gamma$-conjugacy class 
$\mathcal{K}_{e_{j}}=\{ \sigma^{-1}\gamma_{e_{j}}\sigma\,\big|\, \sigma\in\Gamma \}$ of
the generator $\gamma_{e_{j}}$ of $\Gamma_{e_j}$; here, $u(z,w)$ is defined by \eqref{def_u}.
\end{lemma}

\begin{proof}
We start by considering the hyperbolic triangle with vertices 
$e_{j}$, $z$, and $\gamma_{e_{j}}z$, and corresponding angles $2\pi/n_{e_{j}}$, $\pi/2$, and $\pi/2$, 
respectively.  
Bisecting the angle at $e_{j}$ yields a right-angled triangle with angles $\pi/n_{e_{j}}$, $\pi/2$, and $\pi/2$.
By a well-known formula for right-angled triangles (see \cite{Beardon:1995fv}, p.~147), 
we obtain the identity
\begin{align}
\label{rightangled}
\sinh\bigl(d_{\hyp}(e_{j},z)\bigr)
=\sin^{-1}\Bigl(\frac{\pi}{n_{e_{j}}}\Bigr)\sinh\Bigl(\frac{d_{\hyp}(z,\gamma_{e_{j}} z)}{2}\Bigr).
\end{align}
Using the formula $\sinh^{2}(r/2)=(\cosh(r)-1)/2$ and formula \eqref{cosh-hypdist}, this leads to
the identity
\begin{align*}
\sinh^2\bigl(d_{\hyp}(e_{j},z)\bigr)
=\sin^{-2}\Bigl(\frac{\pi}{n_{e_{j}}}\Bigr)
\,\frac{\abs{z-\gamma_{e_{j}} z}^2}{4\Im(z)\Im(\gamma_{e_{j}} z)}\
=\sin^{-2}\Bigl(\frac{\pi}{n_{e_{j}}}\Bigr)\, u(z,\gamma_{e_{j}} z).
\end{align*}
Therefore,
the elliptic Eisenstein series can be written as
\begin{align}
\mathcal{E}^{\mathrm{ell}}_{e_{j}}(z,s)
&=\sin\Bigl(\frac{\pi}{n_{e_{j}}}\Bigr)^{s}
\sum_{\gamma\in\Gamma_{e_{j}}\backslash\Gamma}
  u\bigl(z,\gamma^{-1}\gamma_{e_{j}} \gamma z\bigr)^{-s/2}
\label{eq_eis1},
\end{align}
where $z\in\mathbb{H}$ with $z\neq\gamma e_{j}$ for any $\gamma\in\Gamma$;
note that for $\tilde{\gamma}:=\gamma^{-1}\gamma_{e_{j}} \gamma$ ($\gamma\in\Gamma_{e_{j}}\backslash\Gamma$),
we have
\begin{align*}
u(z,\tilde{\gamma} z)=0
\Longleftrightarrow
z=\tilde{\gamma} z
\Longleftrightarrow
z=\gamma^{-1}e_{j},
\end{align*}
in which case the elliptic Eisenstein series $\mathcal{E}^{\mathrm{ell}}_{e_{j}}(z,s)$ is not defined.
Since the map
$\phi:\Gamma_{e_{j}}\backslash\Gamma\to\mathcal{K}_{e_{j}}$, given by $\Gamma_{e_{j}}\gamma\mapsto\gamma^{-1}\gamma_{e_{j}}\gamma$, is bijective, \eqref{eq_eis1} can finally be rewritten as
\begin{align*}
\mathcal{E}^{\mathrm{ell}}_{e_{j}}(z,s)
=\sin\Bigl(\frac{\pi}{n_{e_{j}}}\Bigr)^{s}
 \sum_{\gamma \in \mathcal{K}_{e_{j}}}
  u(z,\gamma z)^{-s/2}.
\end{align*}
This completes the proof of the lemma.
\end{proof}

\begin{remark}
In \cite{Huber:1956oq}, the following series associated to a pair of hyperbolic
fixed points $h_j$ of $\Gamma$ is studied 
\begin{align*}
\mathcal{G}(z,s)
&= \sum_{\gamma \in \mathcal{K}_{h_{j}}}
  u(z,\gamma z)^{-s}
\end{align*}
with the $\Gamma$-conjugacy class $\mathcal{K}_{h_{j}}=\{ \sigma^{-1}
\gamma_{h_{j}}\sigma\,\big|\, \sigma\in\Gamma\}$ of the generator $\gamma_{h_{j}}$ of the stabilizer group 
$\Gamma_{h_{j}}$ of $h_{j}$ in $\Gamma$. Hence, the elliptic Eisenstein series
can be considered as an elliptic analogue of the series studied in \cite{Huber:1956oq}.
\end{remark}

\begin{lemma}\label{lemma-diff-E}
For $z\in\mathbb{H}$ with $z\neq\gamma e_{j} $ for any $\gamma\in\Gamma$, and
$s\in\mathbb{C}$ with $\Re(s)>1$, the elliptic Eisenstein series 
$\mathcal{E}^{\mathrm{ell}}_{e_{j}}(z,s)$ satisfies the differential equation
$(j=1,\dots, e_{\Gamma})$ 
\begin{align}
\label{eq-diff-E}
\bigl(\Delta_{\hyp}-s(1-s)\bigr)\mathcal{E}^{\mathrm{ell}}_{e_{j}}(z,s)
=-s^{2}\mathcal{E}^{\mathrm{ell}}_{e_{j}}(z,s+2).
\end{align}
\end{lemma}
\begin{proof}
Since the differential operator
\begin{align*}
\Delta_{\hyp}=-\frac{\partial^{2}}{\partial\varrho^{2}}-\frac{1}{\tanh(\varrho)}\frac{\partial}
{\partial\varrho}-\frac{1}{\sinh^{2}(\varrho)}\frac{\partial^{2}}{\partial\vartheta^{2}}
\end{align*}
is invariant under the action of $\Gamma$, it suffices by Lemma \ref{4.2} 
to prove the equality
\begin{align*}
\bigl(\Delta_{\hyp}-s(1-s)\bigr)\sinh(\varrho)^{-s}=-s^{2}\sinh(\varrho)^{-(s+2)}.
\end{align*}
This follows immediately by a straight-forward calculation.
\end{proof} 

The elliptic Eisenstein series $\mathcal{E}^{\mathrm{ell}}_{e_j}(z,s)$ is not a square-integrable
function on $M$. However, there exists an infinite relation to the following square-integrable function.
\begin{lemma}\label{lemma_rel_Pell_Vell}
For $z\in\mathbb{H}$ with $z\neq\gamma e_{j}$ for any $\gamma\in\Gamma$,
and $s\in\mathbb{C}$ with $\Re(s)>1$, we have the relation
\begin{align}\label{rel_Pell_Vell}
\mathcal{E}^{\mathrm{ell}}_{e_j}(z,s)=\sum\limits_{k=0}^{\infty}\frac{(\frac{s}{2})_{k}}{k!}
P^{\mathrm{ell}}_{e_j}(z,s+2k)
\end{align}
with $P^{\mathrm{ell}}_{e_j}(z,s)$ defined by \eqref{eq_ellpoin_P}.
\end{lemma}
\begin{proof}
We first check the absolute and local uniform convergence of the series in the claimed 
relation for fixed $z\in\mathbb{H}$ with $z\neq\gamma e_{j}$ for any $\gamma\in\Gamma$
and $s\in\mathbb{C}$ with $\Re(s)>1$. 
Since $\Gamma$ acts properly discontinously on $\mathbb{H}$ and $z\neq\gamma e_{j}$ 
for any $\gamma\in\Gamma$, the minimal distance 
$\min_{\gamma\in\Gamma}d_{\mathrm{hyp}}(e_{j},\gamma z)$ exists and is strictly positive. 
Using the estimate
\begin{align*}
\cosh\bigl(\varrho(\sigma_{e_{j}}^{-1}\gamma z)\bigr)>C,
\end{align*}
where $C=C(z)>1$ is a positive constant depending on $z$, but which is independent of 
$\gamma\in\Gamma$, together with the estimate $|\Gamma(s')|\leq |\Gamma(\Re(s'))|=\Gamma(\Re(s'))$
for $s'\in\mathbb{C}$ with $\Re(s')>0$, we obtain the bound
\begin{align*}
\sum\limits_{k=0}^{\infty}
\bigg|\frac{(\frac{s}{2})_{k}}{k!}P^{\mathrm{ell}}_{e_j}(z,s+2k)\bigg|
&\leq\sum\limits_{k=0}^{\infty}\frac{(\frac{\Re(s)}{2})_{k}}{k!}
\sum\limits_{\gamma\in\Gamma_{e_{j}}\backslash\Gamma}\cosh\bigl(\varrho(\sigma_{e_j}^{-1}\gamma z)\bigr)
^{-\Re(s)-2k}\\&\leq\sum\limits_{k=0}^{\infty}\frac{(\frac{\Re(s)}{2})_{k}}{k!}C^{-2k}\sum\limits_{\gamma\in
\Gamma_{e_{j}}\backslash\Gamma}\cosh\bigl(\varrho(\sigma_{e_j}^{-1}\gamma z)\bigr)^{-\Re(s)}\\
&=\bigl(1-C^{-2}\bigr)^{-\Re(s)/2}\, P^{\mathrm{ell}}_{e_j}\bigl(z,\Re(s)\bigr).
\end{align*}
This proves that the series in question converges absolutely and locally uniformly for 
$s\in\mathbb{C}$ with $\Re(s)>1$. 

Now, the claimed relation can easily be derived by changing the order of summation, namely 
we compute
\begin{align*}
\sum\limits_{k=0}^{\infty}\frac{(\frac{s}{2})_{k}}{k!}
P^{\mathrm{ell}}_{e_j}(z,s+2k)&=
\sum\limits_{\gamma\in\Gamma_{e_j}\backslash\Gamma}
\cosh\bigl(\varrho(\sigma_{e_{j}}^{-1}\gamma z)\bigr)^{-s}
\sum\limits_{k=0}^{\infty}\frac{(\frac{s}{2})_{k}}{k!}
\cosh\bigl(\varrho(\sigma_{e_{j}}^{-1}\gamma z)\bigr)^{-2k}\\
&=\sum\limits_{\gamma\in\Gamma_{e_j}\backslash\Gamma}
\cosh\bigl(\varrho(\sigma_{e_{j}}^{-1}\gamma z)\bigr)^{-s}
\Bigl(1-\cosh\bigl(\varrho(\sigma_{e_j}^{-1}\gamma z)\bigr)^{-2}\Bigr)^{-s/2}\\
&=\mathcal{E}^{\mathrm{ell}}_{e_j}(z,s).
\end{align*}
This completes the proof of the lemma.
\end{proof}



\section{Meromorphic continuation}
In this section, we establish the meromorphic continuation of the 
elliptic Eisenstein series to the whole $s$-plane, and we determine
its possible poles and residues. The proof relies
on the relation of the elliptic Eisenstein series to the elliptic Poincar\'e series which is 
given in Lemma \ref{lemma_rel_Pell_Vell}. We first state the spectral expansion 
of the elliptic Poincar\'e series which is well-known to the experts.

\begin{proposition}\label{prop-spectral-expansion-Pell}
For $z\in\mathbb{H}$ and $s\in\mathbb{C}$ with $\Re(s)>1$, the elliptic
Poincar\'e series $P^{\mathrm{ell}}_{e_{j}}(z,s)$ associated to the 
elliptic fixed point $e_j\in E_{\Gamma}$ 
admits the spectral expansion
\begin{align}
\label{spectral-expansion-Pell}
P^{\mathrm{ell}}_{e_{j}}(z,s)=
\sum\limits_{r=0}^{\infty}
a_{r,e_{j}}(s)\,\psi_{r}(z)+\frac{1}{4\pi}\sum_{k=1}^{p_{\Gamma}}\,\int\limits_
{-\infty}^{\infty}a_{1/2+it,p_k,e_{j}}(s)
\,{\textstyle\mathcal{E}^{\mathrm{par}}_{p_{k}}\bigl(z,\frac{1}{2}+it\bigr)}dt\,;
\end{align}
here, the coefficients $a_{r,e_{j}}(s)$ and $a_{1/2+it,p_k,e_{j}}(s)$ are 
given by the formulas
\begin{align*}
a_{r,e_{j}}(s)&=
\frac{ 2^{s-1}\sqrt{\pi} }{ n_{e_{j}}\Gamma(s)}\,
\Gamma\Bigl(\frac{s-s_{r}}{2}\Bigr)\Gamma\Bigl(\frac{s-1+s_{r}}{2}\Bigr)
\,\psi_{r}(e_{j}),\\[0.2cm]
a_{1/2+it,p_k,e_{j}}(s)&=
\frac{ 2^{s-1}\sqrt{\pi} }{ n_{e_{j}}\Gamma(s)}\,
\Gamma\Bigl(\frac{s-\frac{1}{2}-it}{2}\Bigr)\Gamma\Bigl(\frac{s-\frac{1}{2}+it}{2}\Bigr)
\,\mathcal{E}^{\mathrm{par}}_{p_{k}}\Bigl(e_{j},\frac{1}{2}-it\Bigr),
\end{align*}
respectively.
For $s\in\mathbb{C}$ with $\Re(s)>1$, the expansion \eqref{spectral-expansion-Pell} 
converges absolutely and uniformly for $z$ ranging over compacta $K\subseteq \mathbb{H}$.
\end{proposition}

\begin{proof}
The proof can be obtained by an explicit computation of the spectral coefficients (see, e.g.,
\cite{Pippich04}). Alternatively, one observes that the Helgason transform
$\widehat{\cosh(\varrho(z))^{-s}}(s_{r})$ of the function $\cosh(\varrho(z))^{-s}$ evaluated at $s_r$
is given by
\begin{align*}
\widehat{\cosh(\varrho(z))^{-s}}(s_{r})=
\frac{ 2^{s-1}\sqrt{\pi}}{\Gamma(s)}
\Gamma\Bigl(\frac{s-s_{r}}{2}\Bigr) \Gamma\Bigl(\frac{s-1+s_{r}}{2}\Bigr),
\end{align*}
hence, the spectral expansion
\eqref{spectral-expansion-Pell} can be obtained as a special case of a so-called non-euclidean 
Poisson summation formula (see, e.g., \cite{Helgason:1984vn}, \cite{Selberg:1956kx}, or \cite{Terras:1985uq}).
\end{proof} 

The explicit knowledge of the spectral coefficients of the 
elliptic Poincar\'e series $P^{\mathrm{ell}}_{e_{j}}(z,s)$ enables us to prove
its meromorphic continuation to the whole $s$-plane, to 
determine its possible poles and to compute its residues.

\begin{proposition}\label{prop-continuation-Pell}
For $z\in\mathbb{H}$, the elliptic Poincar\'e series $P^{\mathrm{ell}}_{e_{j}}(z,s)$ 
admits a meromorphic continuation to the whole $s$-plane.
The possible poles of the function $\Gamma(s)\Gamma(s-1/2)^{-1}P^{\mathrm{ell}}_{e_{j}}(z,s)$ 
are located at the points:
\begin{enumerate}[leftmargin=1mm,itemindent=6mm,listparindent=0em,itemsep=4mm]
\item[\emph{(a)}]
$s=1/2\pm it_{r}-2n$, where $n\in\mathbb{N}$ and 
$\lambda_{r}=s_{r}(1-s_{r})=1/4+t_{r}^{2}$ is the eigenvalue of the eigenfunction 
$\psi_{r}$, which is a simple pole with residue
\begin{align*}
\mathrm{Res}_{s=1/2\pm it_{r}-2n}
\Big[\frac{\Gamma(s)}{\Gamma(s-\frac{1}{2})}\,P^{\mathrm{ell}}_{e_{j}}(z,s)\Big]=
\frac{(-1)^{n}\,2^{1/2\pm it_{r}-2n}\sqrt{\pi}\,\Gamma(\pm it_{r}-n)}
{n!\,n_{e_j}\,\Gamma(\pm it_{r}-2n)}
\sum_{\ell:\,s_{\ell}=s_{r}}
\psi_{\ell}(e_{j})\,\psi_{\ell}(z);
\end{align*}
in case $t_{r}=0$, the factor in front of the sum reduces to
$2^{3/2-2n}\sqrt{\pi}\,(2n)!/(n!)^{2}\,n_{e_j}$.
\item[\emph{(b)}]
$s=1-\rho-2n$, where $n\in\mathbb{N}$ and $w=\rho$ is a pole of the Eisenstein series 
$\mathcal{E}^{\mathrm{par}}_{p_{k}}(z,w)$ with $\Re(\rho)\in(1/2,1]$, 
which is a simple pole with residue
\begin{align*}
&\hspace*{0.5cm}\mathrm{Res}_{s=1-\rho-2n}\Big[\frac{\Gamma(s)}{\Gamma(s-\frac{1}{2})}\,
P^{\mathrm{ell}}_{e_{j}}(z,s)\Big]=
\frac{(-1)^{n}\,2^{1-\rho-2n}\sqrt{\pi}\,\Gamma(\frac{1}{2}-\rho-n)}
{n!\,n_{e_j}\,\Gamma(\frac{1}{2}-\rho-2n)}\times \\[2mm]
&\hspace*{0.5cm}\times\sum\limits_{k=1}^{p_{\Gamma}}\,\Bigl[
\mathrm{Res}_{w=\rho}\,\mathcal{E}^{\mathrm{par}}_{p_{k}}(e_{j},w)\cdot
\mathrm{CT}_{w=\rho}\,\mathcal{E}^{\mathrm{par}}_{p_{k}}(z,1-w)
+
\mathrm{CT}_{w=\rho}\,\mathcal{E}^{\mathrm{par}}_{p_{k}}(e_{j},w)
\cdot\mathrm{Res}_{w=\rho}\,\mathcal{E}^{\mathrm{par}}_{p_{k}}(z,1-w)
\Bigr].
\end{align*}
\item[\emph{(c)}]
$s=\rho-2n$, where $n\in\mathbb{N}$ and $w=\rho$ is a pole of the Eisenstein series 
$\mathcal{E}^{\mathrm{par}}_{p_{k}}(z,w)$ with $\Re(\rho)<1/2$.
If $\rho$ is a simple pole, the residue is given by
\begin{align*}
&\mathrm{Res}_{s=\rho-2n}\Big[\frac{\Gamma(s)}{\Gamma(s-\frac{1}{2})}\,
P^{\mathrm{ell}}_{e_{j}}(z,s)\Big]=
\frac{2^{\rho-2n}\sqrt{\pi}}
{n_{e_j}}
\sum\limits_{\ell=0}^{m}
\frac{(-1)^{\ell}\,(\rho-\frac{1}{2}-2n)_{\ell}}
{\ell!}\times\\[2mm]
&\times\sum\limits_{k=1}^{p_{\Gamma}}\,\Bigl[
\mathrm{CT}_{w=\rho-2(n-\ell)}\,\mathcal{E}^{\mathrm{par}}_{p_{k}}(e_{j},1-w)\cdot
\mathrm{Res}_{w=\rho-2(n-\ell)}\,\mathcal{E}^{\mathrm{par}}_{p_{k}}(z,w)
+
\\[2mm]
&
\phantom{\times\sum\limits_{k=1}^{p_{\Gamma}}\,\Big[}
+
\mathrm{Res}_{w=\rho-2(n-\ell)}\,\mathcal{E}^{\mathrm{par}}_{p_{k}}(e_{j},1-w)
\cdot\mathrm{CT}_{w=\rho-2(n-\ell)}\,\mathcal{E}^{\mathrm{par}}_{p_{k}}(z,w)
\Bigr];
\end{align*}
here $m\in\mathbb{N}$ is such that $-3/2-2m<\Re(s)\leq1/2-2m$. In case
$\Re(s)=1/2-2m$, the summand for $\ell=m$ has to be multiplied by $1/2$.
\end{enumerate}
The poles given in cases (a), (b), (c) might coincide in parts; if this is the case, 
the corresponding residues have to be added accordingly.
\end{proposition}

\begin{proof}
In order to derive the meromorphic continuation of $P^{\mathrm{ell}}_{e_{j}}(z,s)$,
we use the spectral expansion \eqref{spectral-expansion-Pell}. We start by giving 
the meromorphic continuation for the series in \eqref{spectral-expansion-Pell} arising 
from the discrete spectrum. The explicit formula 
\begin{align*}
a_{r,e_{j}}(s)&=\frac{ 2^{s-1}\sqrt{\pi}}{n_{e_{j}}\Gamma(s)}
\Gamma\Bigl(\frac{s-s_{r}}{2}\Bigr) \Gamma\Bigl(\frac{s-1+s_{r}}{2}\Bigr)
\,\psi_{r}(e_{j})
\end{align*}
in terms of $\Gamma$-functions proves the meromorphic continuation for 
the coefficients $a_{r,e_{j}}(s)$ to the whole $s$-plane. 
Writing
\begin{align*}
\Gamma\Bigl(\frac{s-s_{r}}{2}\Bigr)
\Gamma\Bigl(\frac{s-1+s_{r}}{2}\Bigr)=
\Gamma\Bigl(\frac{s-\frac{1}{2}-it_{r}}{2}\Bigr)
\Gamma\Bigl(\frac{s-\frac{1}{2}+it_{r}}{2}\Bigr),
\end{align*}
and applying Stirling's asymptotic formula, we get 
\begin{align*}
\Gamma\Bigl(\frac{s-s_{r}}{2}\Bigr)
\Gamma\Bigl(\frac{s-1+s_{r}}{2}\Bigr)=
\Landau\bigl(t_{r}^{\Re(s)-3/2}e^{-\pi t_{r}/2}\bigr),
\end{align*}
as $t_r\rightarrow\infty$, with an implied constant depending on $s$.
Using this bound together with the well-known sup-norm bound  
\begin{align*}
\sup_{z\in \Gamma\backslash\mathbb{H}}|\psi_{r}(z)|=\Landau\big(\sqrt{t_{r}}\big),
\end{align*}
we find for all, but the finitely many $r$ with $t_r\in[-i/2,0]$ corresponding to
eigenvalues $s_r(1-s_r)=\lambda_{r}<1/4$, the bound
\begin{align*}
a_{r,e_{j}}(s)\,\psi_{r}(z)=\Landau\big(t_{r}^
{\mathrm{Re}(s)-1/2}e^{-\pi t_{r}/2}\big),
\end{align*}
as $t_r\rightarrow\infty$, with an implied constant depending on $s$. 
This
proves that the series in \eqref{spectral-expansion-Pell} arising from the 
discrete spectrum converges absolutely and locally uniformly for all 
$s\in\mathbb{C}$, and hence defines a holomorphic function away from the poles of 
$a_{r,e_{j}}(s)$.\\
The location of the poles of the series in \eqref{spectral-expansion-Pell} arising 
from the discrete spectrum multiplied by the factor
$\Gamma(s)\Gamma(s-1/2)^{-1}$ and the calculation of the residues arising from this 
part is straightforward referring to the fact that $\Gamma(s'/2)$ is a 
meromorphic function for all $s'\in\mathbb{C}$, which has a simple pole at $s'=-2n$ 
($n\in\mathbb{N}$) with residue equal to $2(-1)^n/n!$.

We now turn to give the meromorphic continuation of the integral 
\begin{align}
&\frac{2^{1-s}\pi^{-1/2}\,\Gamma(s)\,n_{e_j}}{4\pi}
\int\limits_{-\infty}^{\infty}
a_{1/2+it,p_k,e_{j}}(s)
\,{\textstyle\mathcal{E}^{\mathrm{par}}_{p_{k}}\bigl(z,\frac{1}{2}+it\bigr)}dt=\notag\\
&\frac{1}{4\pi}\int\limits_{-\infty}^{\infty}
\Gamma\Bigl(\frac{s-\frac{1}{2}-it}{2}\Bigr)\Gamma\Bigl(\frac{s-\frac{1}{2}+it}{2}\Bigr)\,
\mathcal{E}^{\mathrm{par}}_{p_{k}}\Bigl(e_{j},\frac{1}{2}-it\Bigr)
\,\mathcal{E}^{\mathrm{par}}_{p_{k}}\Bigl(z,\frac{1}{2}+it\Bigr)dt
\label{int-Q1}
\end{align}
arising from the continuous part of the spectral expansion 
\eqref{spectral-expansion-Pell} for $k=1,\dots,p_{\Gamma}$
after multiplication by $2^{1-s}\pi^{-1/2}\,\Gamma(s)n_{e_j}$. 
Substituting $t\mapsto 1/2+it$, the integral \eqref{int-Q1} can be 
rewritten as
\begin{align}
\label{int-Q2}
I_{1/2,p_k}(s):=
\frac{1}{4\pi i}\int\limits_{\Re(t)=1/2}
\Gamma\Bigl(\frac{s-t}{2}\Bigr)\Gamma\Bigl(\frac{s-1+t}{2}\Bigr)\,
\mathcal{E}^{\mathrm{par}}_{p_{k}}(e_{j},1-t)
\,\mathcal{E}^{\mathrm{par}}_{p_{k}}(z,t)\,dt\,.
\end{align}
The function $I_{1/2,p_k}(s)$ is holomorphic for $s\in\mathbb{C}$ with 
$\Re(s)>1$, in fact, it is holomorphic for $s\in\mathbb{C}$ satisfying 
$\Re(s)\not=1/2-2n$ ($n\in\mathbb{N}$).\\
Now, let $\varepsilon>0$ be sufficiently small such that 
$\mathcal{E}^{\mathrm{par}}_{p_{k}}(z,s)$ has no poles in the strip 
$1/2-\varepsilon<\Re(s)<1/2+\varepsilon$.
For $s\in\mathbb{C}$ with $1/2<\Re(s)<1/2+\varepsilon$, we then have by the
residue theorem 
\begin{align}
I_{1/2,p_k}(s)&=
I_{1/2+\varepsilon,p_k}(s)-
\frac{1}{2}\,\mathrm{Res}_{t=s}\Big[
\Gamma\Bigl(\frac{s-t}{2}\Bigr)
\Gamma\Bigl(\frac{s-1+t}{2}\Bigr)\,
\mathcal{E}^{\mathrm{par}}_{p_{k}}(e_{j},1-t)
\,\mathcal{E}^{\mathrm{par}}_{p_{k}}(z,t)\Big]\notag\\
&=I_{1/2+\varepsilon,p_k}(s)
+\Gamma\Bigl(s-\frac{1}{2}\Bigr)\,
\mathcal{E}^{\mathrm{par}}_{p_{k}}(e_{j},1-s)
\,\mathcal{E}^{\mathrm{par}}_{p_{k}}(z,s)\,.
\label{int-Q3}
\end{align}
The right-hand side of \eqref{int-Q3} is a meromorphic function for 
$1/2-\varepsilon<\Re(s)<1/2+\varepsilon$ giving the meromorphic continuation 
$I^{(1)}_{1/2,p_k}(s)$ of the integral $I_{1/2,p_k}(s)$ to the strip
$1/2-\varepsilon<\Re(s)<1/2+\varepsilon$. Now, assuming 
$1/2-\varepsilon<\Re(s)<1/2$, and using the residue theorem once again, we
obtain
\begin{align}
I^{(1)}_{1/2,p_k}(s)&=
I_{1/2,p_k}(s)+
\frac{1}{2}\,\mathrm{Res}_{t=1-s}\Big[
\Gamma\Bigl(\frac{s-t}{2}\Bigr)
\Gamma\Bigl(\frac{s-1+t}{2}\Bigr)\,
\mathcal{E}^{\mathrm{par}}_{p_{k}}(e_{j},1-t)
\,\mathcal{E}^{\mathrm{par}}_{p_{k}}(z,t)\Big]+\notag\\
&\phantom{=}\,+\Gamma\Bigl(s-\frac{1}{2}\Bigr)\,
\mathcal{E}^{\mathrm{par}}_{p_{k}}(e_{j},1-s)
\,\mathcal{E}^{\mathrm{par}}_{p_{k}}(z,s)\notag\\
&=I_{1/2,p_k}(s)
+\Gamma\Bigl(s-\frac{1}{2}\Bigr)\,
\mathcal{E}^{\mathrm{par}}_{p_{k}}(e_{j},s)
\,\mathcal{E}^{\mathrm{par}}_{p_{k}}(z,1-s)+
\Gamma\Bigl(s-\frac{1}{2}\Bigr)\,
\mathcal{E}^{\mathrm{par}}_{p_{k}}(e_{j},1-s)
\,\mathcal{E}^{\mathrm{par}}_{p_{k}}(z,s)\,.
\label{int-Q4}
\end{align}
The right-hand side of \eqref{int-Q4} is a meromorphic function for 
$-3/2<\Re(s)<1/2$ giving the meromorphic continuation 
$I^{(2)}_{1/2,p_k}(s)$ of the integral $I^{(1)}_{1/2,p_k}(s)$ to the strip
$-3/2<\Re(s)<1/2$. Summing up, we find that the formulas  \eqref{int-Q3}
and  \eqref{int-Q4} provide the meromorphic continuation of the integral
$I_{1/2,p_k}(s)$ to the strip $-3/2<\Re(s)\leq1/2$.\\
Continuing this two-step process, the meromorphic continuation of the
integral $I_{1/2,p_k}(s)$ to the strip $-3/2-2m<\Re(s)\leq1/2-2m$
($m\in\mathbb{N}$) is given by means of the formula
\begin{align}
I_{1/2,p_k}(s)
&+\sum\limits_{\ell=0}^{m}
\frac{(-1)^{\ell}}{\ell!}
\Gamma\Bigl(s-\frac{1}{2}+\ell\Bigr)\,
\mathcal{E}^{\mathrm{par}}_{p_{k}}(e_{j},1-s-2\ell)
\,\mathcal{E}^{\mathrm{par}}_{p_{k}}(z,s+2\ell)\notag\\
&+\sum\limits_{\ell=0}^{m}
\frac{(-1)^{\ell}}{\ell!}
\Gamma\Bigl(s-\frac{1}{2}+\ell\Bigr)\,
\mathcal{E}^{\mathrm{par}}_{p_{k}}(e_{j},s+2\ell)
\,\mathcal{E}^{\mathrm{par}}_{p_{k}}(z,1-s-2\ell),
\label{int-Q5}
\end{align}
where, on the line $\Re(s)=1/2-2m$, the integral $I_{1/2,p_k}(s)$
has to be replaced by $I_{1/2+\varepsilon,p_k}(s)$ and the summand
for $\ell=m$ in the second sum has to be omitted. In this way, we obtain the meromorphic
continuation of the integral $I_{1/2,p_k}(s)$ to the whole $s$-plane.

Adding up and using the identity
\begin{align}
\label{func_par_eis2}
\sum\limits_{k=1}^{p_{\Gamma}}
\mathcal{E}^{\mathrm{par}}_{p_{k}}(e_{j},1-s')
\,\mathcal{E}^{\mathrm{par}}_{p_{k}}(z,s')
=
\sum\limits_{k=1}^{p_{\Gamma}}
\mathcal{E}^{\mathrm{par}}_{p_{k}}(e_{j},s')
\,\mathcal{E}^{\mathrm{par}}_{p_{k}}(z,1-s'),
\end{align}
which can be derived from the functional equations \eqref{func_par_eis} and \eqref{func_scattering}.
the meromorphic continuation of the continuous part
of the spectral expansion \eqref{spectral-expansion-Pell} after 
multiplication by $\Gamma(s)$, i.e. of
\begin{align*}
&\frac{\Gamma(s)}{4\pi}\sum_{k=1}^{p_{\Gamma}}
\int\limits_{-\infty}^{\infty}
a_{1/2+it,p_{k},e_{j}}(s)\,
{\textstyle\mathcal{E}^{\mathrm{par}}_{p_{k}}\bigl(z,\frac{1}{2}+it\bigr)}dt,
\end{align*}
to the strip $-3/2-2m<\Re(s)\leq1/2-2m$ ($m\in\mathbb{N}$) 
is given by means of the formula
\begin{align}
\frac{2^{s-1}\sqrt{\pi}}{n_{e_j}}\sum\limits_{k=1}^{p_{\Gamma}}I_{1/2,p_k}(s)
\frac{2^{s}\sqrt{\pi}}{n_{e_j}}\sum\limits_{\ell=0}^{m}
\frac{(-1)^{\ell}}{\ell!}
\Gamma\Bigl(s-\frac{1}{2}+\ell\Bigr)\,\sum\limits_{k=1}^{p_{\Gamma}}
\mathcal{E}^{\mathrm{par}}_{p_{k}}(e_{j},1-s-2\ell)
\,\mathcal{E}^{\mathrm{par}}_{p_{k}}(z,s+2\ell)\,,
\label{int-Q6}
\end{align}
where, on the line $\Re(s)=1/2-2m$, the integral $I_{1/2,p_k}(s)$
has to be replaced by $I_{1/2+\varepsilon,p_k}(s)$ and the summand
for $\ell=m$ in the sum over $\ell$ has to be multiplied by the factor $1/2$. 
In this way, we obtain the meromorphic continuation of the continuous 
part of the spectral expansion \eqref{spectral-expansion-Pell} 
multiplied by $\Gamma(s)$ to the whole $s$-plane.

In order to determine the poles arising from the continuous spectrum, 
we work from formula \eqref{int-Q6}, valid in the strip 
$-3/2-2m<\Re(s)\leq1/2-2m$ ($m\in\mathbb{N}$). After dividing by $\Gamma(s-1/2)$,
the only poles can arise from the functions $\mathcal{E}^{\mathrm{par}}_{p_{k}}(e_{j},1-s-2\ell)$ 
and $\mathcal{E}^{\mathrm{par}}_{p_{k}}(z,s+2\ell)$ ($\ell=0,\dots,m$).
The poles arising from $\mathcal{E}^{\mathrm{par}}_{p_{k}}(e_{j},1-s-2\ell)$ 
are located at $s_{\ell}=1-\rho-2\ell$ ($\ell=0,\dots,m$), where $\rho$ is a pole of 
$\mathcal{E}^{\mathrm{par}}_{p_{k}}(e_{j},s)$ with 
$1/2+2(m-\ell)\leq\Re(\rho)<5/2+2(m-\ell)$; therefore, by the results of subsection
\ref{subsec_pareis},
there is only a simple pole for $\ell=m$, i.e. at $s_{m}=1-\rho-2m$,
where $\rho$ is a pole of $\mathcal{E}^{\mathrm{par}}_{p_{k}}(e_{j},s)$ with 
$\rho\in(1/2,1]$. 
The poles arising from $\mathcal{E}^{\mathrm{par}}_{p_{k}}(z,s+2\ell)$ 
are located at $s'_{\ell}=\rho-2\ell$ ($\ell=0,\dots,m$), where $\rho$ is 
a pole of $\mathcal{E}^{\mathrm{par}}_{p_{k}}(z,s)$ with 
$-3/2-2(m-\ell)<\Re(\rho)\leq 1/2-2(m-\ell)$.
The residues can now be derived from formula \eqref{int-Q6}.
This completes the proof of the proposition.
\end{proof} 

\begin{theorem}\label{theo_mero_1}
For $z\in\mathbb{H}$ with $z\neq\gamma e_{j}$ for any $\gamma\in\Gamma$, the 
elliptic Eisenstein series $\mathcal{E}^{\mathrm{ell}}_{e_{j}}(z,s)$ 
associated to the 
elliptic fixed point $e_j\in E_{\Gamma}$ 
admits
a meromorphic continuation the whole $s$-plane.
The possible poles of the 
function $\Gamma(s-1/2)^{-1}\,\mathcal{E}^{\mathrm{ell}}_{e_{j}}(z,s)$ 
are located at the points:
\begin{enumerate}[leftmargin=1mm,itemindent=6mm,listparindent=0em,itemsep=3mm]
\item[\emph{(a)}]
$s=1/2\pm it_{r}-2n$, where $n\in\mathbb{N}$ and 
$\lambda_{r}=s_{r}(1-s_{r})=1/4+t_{r}^{2}$ is the eigenvalue of the eigenfunction 
$\psi_{r}$, which is a simple pole with residue
\begin{align*}
\mathrm{Res}_{s=1/2\pm it_{r}-2n}
\Big[{\textstyle\Gamma(s-\frac{1}{2})^{-1}}\,\mathcal{E}^{\mathrm{ell}}_{e_{j}}(z,s)\Big]=
\frac{(-1)^{n}\,2^{1/2\pm i t_{r}}\sqrt{\pi}\,\Gamma(\pm i t_{r})(\frac{3}{4}\mp\frac{i t_{r}}{2})_{n}^{2}}
{n!\,n_{e_j}\Gamma(\frac{1}{2}\pm i t_{r})\Gamma(\pm i t_{r}-2n)(1\mp i t_{r})_{n}}\,
\sum_{\ell:\,s_{\ell}=s_{r}}
\psi_{\ell}(e_{j})\,\psi_{\ell}(z);
\end{align*}
in case $t_{r}= 0$, the factor in front of the sum reduces to
$(-1)^{n}\,2^{3/2}(2n)!\,(\frac{3}{4})_{n}^{2}/(n!)^2\,n_{e_j}$.
\item[\emph{(b)}]
$s=1-\rho-2n$, where $n\in\mathbb{N}$ and $w=\rho$ is a pole of the Eisenstein series 
$\mathcal{E}^{\mathrm{par}}_{p_{k}}(z,w)$ with $\Re(\rho)\in(1/2,1]$,
which is a simple pole with residue
\begin{align*}
&\mathrm{Res}_{s=1-\rho-2n}\Big[{\textstyle\Gamma(s-\frac{1}{2})^{-1}}
\,\mathcal{E}^{\mathrm{ell}}_{e_{j}}(z,s)\Big]=
\frac{(-1)^{n}\,2^{1-\rho}\sqrt{\pi}\,\Gamma(\frac{1}{2}-\rho)(\frac{1}{2}+\frac{\rho}{2})_{n}^{2}}
{n!\,n_{e_j}\,\Gamma(1 -\rho)\Gamma(\frac{1}{2}-\rho-2n)(\frac{1}{2}+\rho)_{n}}
\times\\[2mm]
&\times\sum\limits_{k=1}^{p_{\Gamma}}\,\Big[
\mathrm{Res}_{w=\rho}\,\mathcal{E}^{\mathrm{par}}_{p_{k}}(e_{j},w)\cdot
\mathrm{CT}_{w=\rho}\,\mathcal{E}^{\mathrm{par}}_{p_{k}}(z,1-w)
+
\mathrm{CT}_{w=\rho}\,\mathcal{E}^{\mathrm{par}}_{p_{k}}(e_{j},w)
\cdot\mathrm{Res}_{w=\rho}\,\mathcal{E}^{\mathrm{par}}_{p_{k}}(z,1-w)
\Big].
\end{align*}
\item[\emph{(c)}]
$s=\rho-2n$, where $n\in\mathbb{N}$ and $w=\rho$ is a pole of the Eisenstein series 
$\mathcal{E}^{\mathrm{par}}_{p_{k}}(z,w)$ with $\Re(\rho)<1/2$. If $\rho$ is
a simple pole, the residue is given by
\begin{align*}
&\mathrm{Res}_{s=\rho-2n}
\Big[{\textstyle\Gamma(s-\frac{1}{2})^{-1}}\,\mathcal{E}^{\mathrm{ell}}_{e_{j}}(z,s)\Big]=
\frac{2\pi}{n_{e_j}\Gamma(\frac{\rho}{2}-n)}
\sum_{k'=0}^{n'}\sum\limits_{\ell=k'}^{m+2k'}
\frac{(-1)^{\ell-k'}\,(\rho-\frac{1}{2}-2n)_{k'+\ell}}
{k'!\,(\ell-k')!\,\Gamma(\frac{\rho}{2}+\frac{1}{2}-n+k')}\times\\[2mm]
&\times\sum\limits_{k=1}^{p_{\Gamma}}\,\Big[
\mathrm{CT}_{w=\rho-2(n-\ell)}\,\mathcal{E}^{\mathrm{par}}_{p_{k}}(e_{j},1-w)\cdot
\mathrm{Res}_{w=\rho-2(n-\ell)}\,\mathcal{E}^{\mathrm{par}}_{p_{k}}(z,w)
+\\[2mm]
&\phantom{\times\sum\limits_{k=1}^{p_{\Gamma}}\,\Big[}
+
\mathrm{Res}_{w=\rho-2(n-\ell)}\,\mathcal{E}^{\mathrm{par}}_{p_{k}}(e_{j},1-w)
\cdot\mathrm{CT}_{w=\rho-2(n-\ell)}\,\mathcal{E}^{\mathrm{par}}_{p_{k}}(z,w)
\Big];
\end{align*}
here $m\in\mathbb{N}$ is such that $-3/2-2m<\Re(s)\leq1/2-2m$, $n'=m$ for 
$-1-2m<\Re(s)\leq1/2-2m$, and $n'=m+1$ for $-3/2-2m<\Re(s)\leq-1-2m$. 
In case $\Re(s)=1/2-2m$, the summand for $\ell=m+2k'$ ($k'=0,\dots,n'$)
has to be multiplied by $1/2$.
\end{enumerate}
The poles given in cases (a), (b), (c) might coincide in parts; if this is the case, 
the corresponding residues have to be added accordingly.
\end{theorem}

\begin{proof}
We start by proving that the function $\mathcal{E}^{\mathrm{ell}}_{e_{j}}(z,s)$ 
has a meromorphic continuation to the half-plane
\begin{align*}
\mathcal{H}_{n}:=\{s\in\mathbb{C}\,|\,\Re(s)>-1-2n\}
\end{align*}
for any $n\in\mathbb{N}$. By Lemma \ref{lemma_rel_Pell_Vell}, which 
states the relation
\begin{align*}
\mathcal{E}^{\mathrm{ell}}_{e_{j}}(z,s)=\sum_{k=0}^{\infty}\frac{(\frac{s}{2})_{k}}{k!}
P^{\mathrm{ell}}_{e_{j}}(z,s+2k),
\end{align*}
we can write
\begin{align}
\mathcal{E}^{\mathrm{ell}}_{e_{j}}(z,s)
&=\sum_{k=0}^{n}\frac{(\frac{s}{2})_{k}}{k!}
P^{\mathrm{ell}}_{e_{j}}(z,s+2k)
+\sum_{k=n+1}^{\infty}\frac{(\frac{s}{2})_{k}}{k!}
P^{\mathrm{ell}}_{e_{j}}(z,s+2k).
\label{split-E}
\end{align}
Now, we prove that the series 
\begin{align}
\sum_{k=n+1}^{\infty}\frac{(\frac{s}{2})_{k}}{k!}
P^{\mathrm{ell}}_{e_{j}}(z,s+2k)
\label{series}
\end{align}
is a holomorphic function on the half-plane $\mathcal{H}_{n}$.
For this we prove the absolute and locally uniform convergence of this series
for fixed $z\in\mathbb{H}$ with $z\neq\gamma e_j$ 
for any $\gamma\in\Gamma$ and $s\in\mathbb{C}$ with $\Re(s)>-1-2n$. 
Since $\Gamma$ acts properly discontinously on $\mathbb{H}$ and $z\neq\gamma e_j$ 
for any $\gamma\in\Gamma$, the minimal distance $\min_{\gamma\in\Gamma}d_{\hyp}(e_j,\gamma z)$
exists and is strictly positive. Using the estimate
\begin{align*}
\cosh\bigl(\varrho(\sigma_{e_{j}}^{-1}\gamma z)\bigr)>C,
\end{align*}
where $C=C(z)>1$ is a positive constant depending on $z$, but which is independent of 
$\gamma\in\Gamma$, 
together with the estimate $|\Gamma(s')|\leq |\Gamma(\Re(s'))|=\Gamma(\Re(s'))$
for $s'\in\mathbb{C}$ with $\Re(s')>0$, we obtain the bound
\begin{align*}
&\sum_{k=n+1}^{\infty}\bigg|\frac{(\frac{s}{2})_{k}}{k!}
P^{\mathrm{ell}}_{e_{j}}(z,s+2k)\bigg|
=\sum_{k=0}^{\infty}\bigg|\frac{(\frac{s}{2})_{k+n+1}}{(k+n+1)!}
P^{\mathrm{ell}}_{e_{j}}(z,s+2(k+n+1))\bigg|\\
&\leq
P^{\mathrm{ell}}_{e_{j}}\bigl(z,\Re(s)+2(n+1)\bigr)\,
\sum_{k=0}^{\infty}\frac{\big|\bigl(\frac{\Re(s)}{2}\bigr)_{k+n+1}\big|}{(k+n+1)!}C^{-2k}.
\end{align*}
From this, we derive that the series \eqref{series}
converges absolutely and locally uniformly for $s\in\mathbb{C}$ with 
$\Re(s)>-1-2n$, which proves that 
is a holomorphic function on the half-plane $\mathcal{H}_{n}$. 
Since the finite sum 
\begin{align*}
\sum_{k=0}^{n}\frac{(\frac{s}{2})_{k}}{k!}P^{\mathrm{ell}}_{e_{j}}(z,s+2k)
\end{align*}
in \eqref{split-E} is a meromorphic function on the whole $s$-plane by Proposition
\ref{prop-continuation-Pell}, we conclude that $\mathcal{E}^{\mathrm{ell}}_{e_{j}}(z,s)$ 
has a meromorphic continuation to the half-plane $\mathcal{H}_{n}$. 
Since $n$ was chosen arbitrarily, this proves the meromorphic continuation of 
$\mathcal{E}^{\mathrm{ell}}_{e_{j}}(z,s)$ to the whole $s$-plane.

In order to determine the poles of $\mathcal{E}^{\mathrm{ell}}_{e_{j}}(z,s)$, 
we calculate its poles in the strip
\begin{align*}
\mathcal{S}_{n}:=\{s\in\mathbb{C}\,|\,-1-2n<\Re(s)\leq 1-2n\}
\end{align*}
for any $n\in\mathbb{N}$. By considering $\mathcal{E}^{\mathrm{ell}}_{e_{j}}(z,s)$ 
with its decomposition \eqref{split-E} in the strip $\mathcal{S}_{n}$, we see that 
the possible poles arise from the finite sum 
\begin{align*}
F_{n}(z,s):=\sum_{k=0}^{n}\frac{(\frac{s}{2})_{k}}{k!}
P^{\mathrm{ell}}_{e_{j}}(z,s+2k).
\end{align*}
Therefore, by Proposition \ref{prop-continuation-Pell}, the possible poles 
of the function $F_{n}(z,s)$ in the strip $\mathcal{S}_{n}$ are located at the points 
$s=1/2\pm i t_{r}-2n$, where $\lambda_{r}=1/4+t_{r}^{2}$ is the eigenvalue 
of the eigenfunction $\psi_{r}$, at the points $s=1-\rho-2n$, where 
$w=\rho$ is a pole of the Eisenstein series 
$\mathcal{E}^{\mathrm{par}}_{p_{k}}(z,w)$ with $\Re(\rho)\in(1/2,1]$,
and at the points $s=\rho-2n'$, where $w=\rho$ is a pole of the Eisenstein series 
$\mathcal{E}^{\mathrm{par}}_{p_{k}}(z,w)$ with $\Re(\rho)<1/2$ and where $n'\in\mathbb{N}$ 
satisfying $-1-2n<\Re(\rho)-2n'\leq 1-2n$.

We now turn to compute the residues of the function 
$\Gamma(s-1/2)^{-1}\mathcal{E}^{\mathrm{ell}}_{e_{j}}(z,s)$,
at the possible poles in the strip $\mathcal{S}_{n}$ for any $n\in\mathbb{N}$.
To do this, we write
\begin{align}
&
\frac{F_{n}(z,s)}{\Gamma(s-\frac{1}{2})}=
\frac{2\sqrt{\pi}}{\Gamma(\frac{s}{2})\Gamma(s-\frac{1}{2})}
\sum_{k=0}^{n}
\frac{2^{-(s+2k)}\,\Gamma(s-\frac{1}{2}+2k)}
{k!\,\Gamma(\frac{s}{2}+\frac{1}{2}+k)}
\frac{\Gamma(s+2k)}
{\Gamma(s-\frac{1}{2}+2k)}\,P^{\mathrm{ell}}_{e_{j}}(z,s+2k),
\label{eq_Fn}
\end{align}
where we used the equality 
$$
{\textstyle
\Gamma\bigl(\frac{s}{2}+k\bigr)}
=\frac{2^{1-(s+2k)}\sqrt{\pi}\,\Gamma(s+2k)}{\Gamma(\frac{s}{2}+\frac{1}{2}+k)},
$$
which follows by means of the duplication formula \eqref{eq-duplication-Gamma}.
The explicit formula for the residues of the function
$\Gamma(s)\Gamma(s-1/2)^{-1}P^{\mathrm{ell}}_{e_{j}}(z,s)$ 
given in Proposition \ref{prop-continuation-Pell} (a) now leads 
to the following residue of the function $\Gamma(s-1/2)^{-1}F_{n}(z,s)$ 
at $s=s_r-2n=1/2+i t_{r}-2n$ ($t_{r}\not= 0$)
\begin{align*}
\res_{s=s_{r}-2n}\Big[
\frac{F_{n}(z,s)}{\Gamma(s-\frac{1}{2})}\Big]
&=\frac{(-1)^{n}\,2\pi}{n_{e_j}\Gamma(\frac{s_{r}}{2}-n)\Gamma(s_{r}-\frac{1}{2}-2n)}
\sum_{k=0}^{n}\frac{(-1)^{k}\,\Gamma(s_{r}-\frac{1}{2}-n+k)}
{k!(n-k)!\,\Gamma(\frac{s_{r}}{2}+\frac{1}{2}-n+k)}
\sum_{\ell:\,s_{\ell}=s_{r}}
\psi_{\ell}(e_{j})\,\psi_{\ell}(z).
\end{align*}
Now, applying formula 
\begin{align*}
(a+b)_{n}=\sum\limits_{k=0}^{n}(-1)^{k}{n \choose k}(-b)_{k}(a+k)_{n-k}
\end{align*}
with $a:=s_r/2+1/2-n$ and $b:=-s_r+1/2+n$, 
adding up, and using the equality
\begin{align*}
\frac{\Gamma(s_{r}-\frac{1}{2}-n)}
{\Gamma(\frac{s_{r}}{2}+\frac{1}{2})\Gamma(\frac{s_{r}}{2}-n)}=
\frac{2^{s_{r}-1}\,\Gamma(s_{r}-\frac{1}{2})\,(1-\frac{s_{r}}{2})_{n}  }
{\sqrt{\pi}\,\Gamma(s_{r})(\frac{3}{2}-s_{r})_{n}},
\end{align*}
which can be derived by using twice formula
\eqref{eq_gamma_minusn} and then duplication formula \eqref{eq-duplication-Gamma},
we obtain
\begin{align*}
\res_{s=s_{r}-2n}\Big[\frac{F_{n}(z,s)}{\Gamma(s-\frac{1}{2})}\Big]
&=\frac{(-1)^{n}\,2^{s_{r}}\sqrt{\pi}\,\Gamma(s_{r}-\frac{1}{2})(1-\frac{s_{r}}{2})_{n}^{2}}
{n!\,n_{e_j}\,\Gamma(s_{r})\Gamma(s_{r}-\frac{1}{2}-2n)(\frac{3}{2}-s_{r})_{n} }\,
\sum_{\ell:\,s_{\ell}=s_{r}}
\psi_{\ell}(e_{j})\,\psi_{\ell}(z).
\end{align*}
The residue of the function $\Gamma(s-1/2)^{-1}F_{n}(z,s)$ at $s=1/2-i t_{r}-2n$ ($t_{r}\not= 0$)
is computed similarly replacing $s_r=1/2+i t_{r}$ by $1-s_r=1/2-i t_{r}$.
Summing up, the residue of the function $\Gamma(s-1/2)^{-1}F_{n}(z,s)$ at $s=1/2\pm i t_{r}-2n$ ($t_{r}\not= 0$)
is given by
\begin{align*}
\res_{s=1/2\pm i t_{r}-2n}\Big[\frac{F_{n}(z,s)}{\Gamma(s-\frac{1}{2})}\Big]=
\frac{(-1)^{n}\,2^{1/2\pm i t_{r}}\sqrt{\pi}\,\Gamma(\pm i t_{r})(\frac{3}{4}\mp\frac{i t_{r}}{2})_{n}^{2}}
{n!\,n_{e_j}\Gamma(\frac{1}{2}\pm i t_{r})\Gamma(\pm i t_{r}-2n)(1\mp i t_{r})_{n}}\,
\sum_{\ell:\,s_{\ell}=s_{r}}
\psi_{\ell}(e_{j})\,\psi_{\ell}(z).
\end{align*}

The residues in the remaining cases are computed similarly.
This completes the proof of the theorem.
\end{proof} 

An immediate consequence of Theorem \ref{theo_mero_1} is the following corollary.
\begin{corollary}
For $z\in\mathbb{H}$ with $z\neq\gamma e_{j}$ for any $\gamma\in\Gamma$, the 
elliptic Eisenstein series $\mathcal{E}^{\mathrm{ell}}_{e_{j}}(z,s)$ 
associated to the elliptic fixed point $e_j\in E_{\Gamma}$ 
admits a simple pole at $s=1$ with residue
\begin{align*}
&\mathrm{Res}_{s=1}
\mathcal{E}^{\mathrm{ell}}_{e_{j}}(z,s)=
\frac{2\pi}{n_{e_j}\,\vol_{\hyp}(\mathcal{F}_{\Gamma})}.
\end{align*}
\end{corollary}

\section{A Kronecker limit type formula}
In this section, we study the behaviour of the elliptic Eisenstein 
series at $s=0$  for an arbitrary Fuchsian subgroup
$\Gamma\subset\PSL_2(\mathbb{R})$ of the first kind.

\begin{proposition}\label{prop_laurent}
For $z\in\mathbb{H}$ with $z\neq\gamma e_{j}$ for any $\gamma\in\Gamma$, 
the elliptic Eisenstein series $\mathcal{E}^{\mathrm{ell}}_{e_{j}}(z,s)$ 
associated to the elliptic fixed point $e_j\in E_{\Gamma}$ 
admits a Laurent expansion at $s=0$ of the form
\begin{align*}
\mathcal{E}^{\mathrm{ell}}_{e_{j}}(z,s)-
\frac{2^{s}\sqrt{\pi}\,\Gamma(s-\frac{1}{2})}{n_{e_j}\,\Gamma(s)}
\sum\limits_{k=1}^{p_{\Gamma}}
\mathcal{E}^{\mathrm{par}}_{p_{k}}(e_{j},1-s)
\,\mathcal{E}^{\mathrm{par}}_{p_{k}}(z,s)=
-\frac{2\pi}{n_{e_j}\vol_{\hyp}(\mathcal{F}_{\Gamma})}
+\mathcal{K}_{e_j}(z) \cdot s+\Landau(s^2),
\end{align*}
where $\mathcal{K}_{e_j}(z)$ is a real-valued function, which fulfills $n_{e_j}\mathcal{K}_{e_j}(z)
=n_{z}\mathcal{K}_{z}(e_j)$
and which is invariant
with respect to $\Gamma$.
Moreover, for any $\gamma\in\Gamma$, we have the estimate
\begin{align}
\mathcal{K}_{e_j}(z)=-\log|z-\gamma e_{j}|+\Landau(1)
\label{eq_estimate_kron}
\end{align}
as $z\to\gamma e_{j}$.
\end{proposition}

\begin{proof}
For $s\in\mathbb{C}$ with $\Re(s)>-1$, by the proof of Theorem \ref{theo_mero_1},
the elliptic Eisenstein series $\mathcal{E}^{\mathrm{ell}}_{e_{j}}(z,s)$ is given by 
decomposition \eqref{split-E} with $n=0$,
i.e., we have
\begin{align*}
\mathcal{E}^{\mathrm{ell}}_{e_{j}}(z,s)&=
P^{\mathrm{ell}}_{e_{j}}(z,s)
+\sum\limits_{k=1}^{\infty}\frac{(\frac{s}{2})_{k}}{k!}P^{\mathrm{ell}}_{e_{j}}(z,s+2k).
\end{align*}
Hence, introducing the notation
\begin{align}\label{def_rest_pareis}
R_{e_j}(z,s):=
\frac{2^{s}\sqrt{\pi}\,\Gamma(s-\frac{1}{2})}{n_{e_j}\,\Gamma(s)}
\sum\limits_{k=1}^{p_{\Gamma}}
\mathcal{E}^{\mathrm{par}}_{p_{k}}(e_{j},1-s)
\,\mathcal{E}^{\mathrm{par}}_{p_{k}}(z,s),
\end{align}
we have for $s\in\mathbb{C}$ with $\Re(s)>-1$ the identity
%
\begin{align}
\mathcal{E}^{\mathrm{ell}}_{e_{j}}(z,s)-R_{e_j}(z,s)
=P^{\mathrm{ell}}_{e_{j}}(z,s)-R_{e_j}(z,s)
+\sum\limits_{k=1}^{\infty}\frac{(\frac{s}{2})_{k}}{k!}P^{\mathrm{ell}}_{e_{j}}(z,s+2k).
\label{eq_deco2}
\end{align}
In the first step, we determine the Laurent expansion of the function
$P^{\mathrm{ell}}_{e_{j}}(z,s)-R_{e_j}(z,s)$
at $s=0$.
By the proof of Proposition \ref{prop-continuation-Pell}, we derive
for 
$s\in\mathbb{C}$ with $-3/2<\Re(s)< 1/2$ the identity
\begin{align}
P^{\mathrm{ell}}_{e_{j}}(z,s)-R_{e_j}(z,s)
=\sum\limits_{r=0}^{\infty}a_{r,e_{j}}(s)\,\psi_{r}(z)
+\frac{1}{4\pi}\,\sum\limits_{k=1}^{p_{\Gamma}}\int\limits_{-\infty}^{\infty}
a_{1/2+it,p_{k},e_{j}}(s)\,
{\textstyle
\mathcal{E}^{\mathrm{par}}_{p_{k}}(z,\frac{1}{2}+it)}\,dt,
\label{eq_spectral_F}
\end{align}
where the coefficients $a_{r,e_{j}}(s)$ resp. $a_{1/2+it,p_k,e_{j}}(s)$ are 
explicitly given in Proposition \ref{prop-continuation-Pell}.
For $r=0$, i.e., $t_r=-i/2$ and $s_r=1$, the function 
$$
a_{0,e_{j}}(s)\,\psi_{0}(z)=
\frac{ 2^{s-1}\sqrt{\pi} }{ n_{e_{j}}\vol_{\hyp}(\mathcal{F}_{\Gamma})\,\Gamma(s)}\,
\Gamma\Bigl(\frac{s-1}{2}\Bigr)\Gamma\Bigl(\frac{s}{2}\Bigr)
$$ 
in the series arising from the discrete spectrum in \eqref{eq_spectral_F} 
admits a Laurent expansion at $s=0$ of the form
\begin{align*}
a_{0,e_{j}}(s)\,\psi_{0}(z)=
-\frac{ 2\pi}{n_{e_j}\vol_{\hyp}(\mathcal{F}_{\Gamma})}
-\frac{ 2 \pi }{n_{e_j}\vol_{\hyp}(\mathcal{F}_{\Gamma})}\cdot s
+\Landau(s^2).
\end{align*}
For $r>0$, the function $a_{r,e_{j}}(s)\,\psi_{r}(z)$ 
admits a Laurent expansion at $s=0$ of the form
\begin{align*}
a_{r,e_{j}}(s)\,\psi_{r}(z)
&=\frac{\sqrt{\pi} }{2\,n_{e_{j}}}\,
\Gamma\Bigl(\frac{-s_{r}}{2}\Bigr)\Gamma\Bigl(\frac{s_{r}-1}{2}\Bigr)\,
\psi_{r}(e_{j})\psi_{r}(z)\cdot s +\Landau(s^2).
\end{align*}
Furthermore, for $s_r=1/2+it_{r}$ with $t_r\in(-i/2,0]$, the
term $\Gamma(-s_r/2) \Gamma((s_r-1)/2)$ is real-valued, and,
for $s_r=1/2+it_{r}$ with $t_r>0$, we have
$$
\Gamma\Bigl(\frac{-s_{r}}{2}\Bigr)\Gamma\Bigl(\frac{s_{r}-1}{2}\Bigr)=
\Big|\Gamma\Bigl(\frac{-\frac{1}{2}+it_{r}}{2}\Bigr)\Big|^2,
$$
which again is real-valued.
Next, the function $a_{1/2+it,p_{k},e_{j}}(s)$ ($k=1,\dots,p_{\Gamma}$)
in the integral arising from the continuous spectrum in \eqref{eq_spectral_F}
admits a Laurent expansion at $s=0$ of the form
\begin{align*}
a_{1/2+it,p_{k},e_{j}}(s)
&= \frac{\sqrt{\pi}}{2\,n_{e_{j}}}\,
\Big|\Gamma\Bigl(\frac{-\frac{1}{2}+it}{2}\Bigr)\Big|^2
\mathcal{E}^{\mathrm{par}}_{p_{k}}\Bigl(e_{j},\frac{1}{2}-it\Bigr)\cdot s +\Landau(s^2).
\end{align*}
Furthermore, using identity \eqref{func_par_eis2} together
with $\overline{\mathcal{E}^{\mathrm{par}}_{p_{k}}}(\cdot,1/2+it)
=\mathcal{E}^{\mathrm{par}}_{p_{k}}(\cdot,1/2-it)$
for $t\in\mathbb{R}$,
we deduce that
\begin{align*}
S_{e_{j}}(z):&=\sum\limits_{k=1}^{p_{\Gamma}}\int\limits_{-\infty}^{\infty}
\frac{\sqrt{\pi}}{2\,n_{e_{j}}}\,
\Big|\Gamma\Bigl(\frac{-\frac{1}{2}+it}{2}\Bigr)\Big|^2\,
\mathcal{E}^{\mathrm{par}}_{p_{k}}\Bigl(e_{j},\frac{1}{2}-it\Bigr)\,
\mathcal{E}^{\mathrm{par}}_{p_{k}}\Bigl(z,\frac{1}{2}+it\Bigr)dt\\
&=\sum\limits_{k=1}^{p_{\Gamma}}\int\limits_{-\infty}^{\infty}
\frac{\sqrt{\pi}}{2\,n_{e_{j}}}\,
\Big|\Gamma\Bigl(\frac{-\frac{1}{2}+it}{2}\Bigr)\Big|^2\,
\overline{\mathcal{E}^{\mathrm{par}}_{p_{k}}}\Bigl(e_{j},\frac{1}{2}-it\Bigr)\,
\overline{\mathcal{E}^{\mathrm{par}}_{p_{k}}}\Bigl(z,\frac{1}{2}+it\Bigr)dt
=\overline{S_{e_{j}}(z)},
\end{align*}
hence, $S_{e_{j}}(z)$ is also real-valued.
All in all, we obtain a Laurent expansion at $s=0$ of the form
\begin{align*}
&P^{\mathrm{ell}}_{e_{j}}(z,s)-R_{e_j}(z,s)
=
-\frac{2\pi}{n_{e_j}\vol_{\hyp}(\mathcal{F}_{\Gamma})}
+F_{e_j}(z) \cdot s+\Landau(s^2),
\end{align*}
where 
\begin{align}
\notag
F_{e_j}(z)&:=
-\frac{ 2 \pi }{n_{e_j}\vol_{\hyp}(\mathcal{F}_{\Gamma})}+
\frac{\sqrt{\pi} }{2\,n_{e_{j}}}\,
\sum\limits_{r>0}^{\infty}
\Gamma\Bigl(\frac{-s_{r}}{2}\Bigr)\Gamma\Bigl(\frac{s_{r}-1}{2}\Bigr)\,\psi_{r}(e_{j})\psi_{r}(z)+\\
&\phantom{:=\,}\,+\frac{1}{8\sqrt{\pi}\,n_{e_{j}}}\,\sum\limits_{k=1}^{p_{\Gamma}}\int\limits_{-\infty}^{\infty}
\Big|\Gamma\Bigl(\frac{-\frac{1}{2}+it}{2}\Bigr)\Big|^2
\mathcal{E}^{\mathrm{par}}_{p_{k}}\Bigl(e_{j},\frac{1}{2}-it\Bigr)
\,
\mathcal{E}^{\mathrm{par}}_{p_{k}}\Bigl(z,\frac{1}{2}+it\Bigr)\,dt
\label{eq_def_F}
\end{align}
is a real-valued function,  which clearly fulfills $n_{e_{j}}F_{e_j}(z)=n_{z}F_{z}(e_j)$
and which is invariant with respect to $\Gamma$.

In the second step, we determine the Laurent expansion of the 
series
\begin{align*}
\sum\limits_{k=1}^{\infty}\frac{(\frac{s}{2})_{k}}{k!}
P^{\mathrm{ell}}_{e_{j}}(z,s+2k)
\end{align*}
in \eqref{eq_deco2} at $s=0$. 
By the proof of Theorem \ref{theo_mero_1}, this series is a holomorphic function 
for $z\in\mathbb{H}$ with $z\neq\gamma e_{j}$ for any $\gamma\in\Gamma$, 
and $s\in\mathbb{C}$ with $\Re(s)>-1$. Since the functions $P^{\mathrm{ell}}_{e_{j}}(z,s+2k)$ 
($k\in\mathbb{N}$, $k\geq1$) are non-vanishing and the Pochhammer symbol
admits a Laurent expansion at $s=0$ of the form
\begin{align*}
\Bigl(\frac{s}{2}\Bigr)_{k}=
\frac{(k-1)!}{2}\cdot s+\Landau(s^{2}),
\end{align*}
we have the following Laurent expansion at $s=0$ 
\begin{align*}
\sum\limits_{k=1}^{\infty}\frac{(\frac{s}{2})_{k}}{k!}
P^{\mathrm{ell}}_{e_{j}}(z,s+2k)
=G_{e_j}(z)\cdot s+ \Landau(s^{2})
\end{align*}
with
\begin{align}
G_{e_j}(z):=\sum\limits_{k=1}^{\infty}
\frac{1}{2k} P^{\mathrm{ell}}_{e_{j}}(z,2k).
\label{eq_def_G}
\end{align}
For $z\in\mathbb{H}$ with $z\neq\gamma e_{j}$ for any $\gamma\in\Gamma$, 
the function $G_{e_j}(z)$ is a real-valued function, which fulfills
$n_{e_{j}}G_{e_j}(z)=n_{z}G_{z}(e_j)$ and which is
invariant with respect to $\Gamma$, since the function
$P^{\mathrm{ell}}_{e_{j}}(z,2k)$ ($k\in\mathbb{N}$, $k\geq1$)
is real-valued, fulfills
$n_{e_{j}}P^{\mathrm{ell}}_{e_{j}}(z,2k)=n_{z}P^{\mathrm{ell}}_{z}(e_{j},2k)$, and is invariant with respect to $\Gamma$, 
by the very definition of the series. 

Summing up, and letting 
\begin{align}
\mathcal{K}_{e_j}(z):=F_{e_j}(z)+G_{e_j}(z)
\label{eq_summe_K}
\end{align}
for $z\in\mathbb{H}$ with $z\neq\gamma e_{j}$ for 
any $\gamma\in\Gamma$, 
this proves the asserted Laurent expansion at $s=0$.

We are left to prove the estimate \eqref{eq_estimate_kron}.
From \eqref{eq_def_F} 
we immediately derive the bound
$F_{e_j}(z)=\Landau(1)$ as $z\to\gamma e_{j}$ for any $\gamma\in\Gamma$.
Therefore, by \eqref{eq_summe_K}, it remains
to prove the estimate
\begin{align*}
G_{e_j}(z)=-\log|z-\gamma e_{j}|+\Landau(1)
\end{align*}
as $z\to\gamma e_{j}$ for any $\gamma\in\Gamma$. To do this we fix some $\gamma\in\Gamma/\Gamma_{e_j}$.
We choose 
$\varepsilon=\varepsilon(e_{j},\Gamma)\in\mathbb R_{>0}$ sufficiently 
small such that 
$$
\bigl\{ \gamma'\in\Gamma \,\big|\, 
\gamma' \mathcal{B}_{2\varepsilon}(\gamma e_{j})\cap 
\mathcal{B}_{2\varepsilon}(\gamma e_{j})\not=\emptyset\bigr\}=\Gamma_{\gamma e_{j}},
$$
where
$\mathcal{B}_{2\varepsilon}(\gamma e_{j})
:=\{w\in\mathbb{H}\,|\,\varrho(\sigma_{e_{j}}^{-1}\gamma^{-1}w)<2\varepsilon\}$
denotes the open hyperbolic disc of radius $2\varepsilon$ centered at $\gamma e_j$.
Without loss of generality we may assume that $z\in\mathbb{H}$, $z\not=\gamma e_{j}$,
is sufficiently close to $\gamma e_{j}$, say
$$
d_{\hyp}(\gamma e_j, z)=\varrho(\sigma_{e_{j}}^{-1}\gamma^{-1}z)<\varepsilon.
$$
We note that for any $\gamma'\in\Gamma/\Gamma_{e_j}$ with $\gamma'\not=\gamma$, 
we have $d_{\hyp}(\gamma' e_j,\gamma e_j)>2\varepsilon$,
which, by the triangle equality 
$d_{\hyp}(\gamma' e_j, z)+d_{\hyp}(\gamma e_j, z)\geq
d_{\hyp}(\gamma' e_j,\gamma e_j)$,
gives the bound
\begin{align}
d_{\hyp}(\gamma' e_j,z)=
\varrho(\sigma_{e_{j}}^{-1}\gamma'^{-1} z)>2\varepsilon-\varepsilon=\varepsilon>
\varrho(\sigma_{e_{j}}^{-1}\gamma^{-1} z).
\label{eq_boundl}
\end{align}
Introducing the notation $C:=\cosh\bigl(\varrho(\sigma_{e_{j}}^{-1}\gamma^{-1} z)\bigr)$,
we write
\begin{align*}
G_{e_j}(z)&
=
\sum\limits_{k=0}^{\infty}\frac{1}{2k+2}
\sum_{\gamma'\in\Gamma/\Gamma_{e_j}}
\cosh\bigl(\varrho(\sigma_{e_{j}}^{-1}\gamma'^{-1} z)\bigr)^{-(2k+2)}\\
&=-\frac{1}{2}\log\bigl(1-C^{-2}\bigr)+
\sum\limits_{k=0}^{\infty}\frac{1}{2k+2}
\sum_{\substack{\gamma'\in\Gamma/\Gamma_{e_j}\\\gamma'\not=\gamma}}
\cosh\bigl(\varrho(\sigma_{e_{j}}^{-1}\gamma'^{-1} z)\bigr)^{-(2k+2)},
\end{align*}
where for the last equality we used the identity
\begin{align}
\sum\limits_{k=0}^{\infty}\frac{C^{-(2k+2)}}{2k+2}=
-\frac{1}{2}\log\bigl(1-C^{-2}\bigr),
\label{eq_identity_log}
\end{align}
keeping in mind that $C>1$. 
Now, using formula \eqref{cosh-hypdist}, namely 
\begin{align*}
\cosh\bigl(\varrho(\sigma_{e_{j}}^{-1}\gamma^{-1} z)\bigr)=
\cosh\bigl(d_{\hyp}(\gamma e_j, z)\bigr)=
1+\frac{\abs{z-\gamma e_{j}}^2}{2\Im(z)\Im(\gamma e_{j})},
\end{align*}
we derive the identity
\begin{align*}
-\frac{1}{2}\log\Bigl(\sinh\bigl(\varrho(\sigma_{e_{j}}^{-1}\gamma^{-1} z)\bigr)^2\Bigr)=
-\frac{1}{2}\log\Bigl(\frac{\abs{z-\gamma e_{j}}^2}{\Im(z)\Im(\gamma e_{j})}\Bigr)-\frac{1}{2}
\log\Bigl(1+\frac{\abs{z-\gamma e_{j}}^2}{4\Im(z)\Im(\gamma e_{j})} \Bigr).
\end{align*}
This leads to the estimate
\begin{align*}
-\frac{1}{2}\log\bigl(1-C^{-2}\bigr)&=
-\frac{1}{2}\log\Bigl(\sinh\bigl(\varrho(\sigma_{e_{j}}^{-1}\gamma^{-1} z)\bigr)^2\Bigr)
+\log\Bigl(\cosh\bigl(\varrho(\sigma_{e_{j}}^{-1}\gamma^{-1} z)\bigr)\Bigr)\\
&= -\log\abs{z-\gamma e_{j}}+ \Landau(1)
\end{align*}
as $z\to\gamma e_{j}$. Finally, using the bound \eqref{eq_boundl} and the identity \eqref{eq_identity_log} with
$C=\cosh(\varepsilon)$, we get
\begin{align*}
\biggl|\sum\limits_{k=0}^{\infty}\frac{1}{2k+2}
\sum_{\substack{\gamma'\in\Gamma/\Gamma_{e_j}\\\gamma'\not=\gamma}}
\cosh\bigl(\varrho(\sigma_{e_{j}}^{-1}\gamma'^{-1} z)\bigr)^{-(2k+2)}\biggr|
&\leq\cosh(\varepsilon)^{2}\log\bigl(\tanh(\varepsilon)^{-1}\bigr)\sum_{\substack{\gamma'\in\Gamma/\Gamma_{e_j}\\\gamma'\not=\gamma}}
\cosh\bigl(\varrho(\sigma_{e_{j}}^{-1}\gamma'^{-1} z)\bigr)^{-2}\\[2mm]
&\leq\cosh(\varepsilon)^{2}\log\bigl(\tanh(\varepsilon)^{-1}\bigr)P^{\mathrm{ell}}_{e_{j}}(z,2).
\end{align*}
Therefore, we get the bound
\begin{align*}
&\sum\limits_{k=0}^{\infty}\frac{1}{2k+2}
\sum_{\substack{\gamma'\in\Gamma/\Gamma_{e_j}\\\gamma'\not=\gamma}}
\cosh\bigl(\varrho(\sigma_{e_{j}}^{-1}\gamma' z)\bigr)^{-(2k+2)}
=\Landau(1)
\end{align*}
as $z\to\gamma e_{j}$. Hence, adding up, we have proven the asserted estimate
\begin{align*}
G_{e_j}(z)
= -\log\abs{z-\gamma e_{j}}+ \Landau(1)
\end{align*}
as $z\to\gamma e_{j}$.
This completes the proof of the proposition.
\end{proof}

By standard arguments (see, e.g., \cite{Fay:1977pq}), we can now establish the following Kronecker limit type formula
for elliptic Eisenstein series.

\begin{theorem}\label{theo_kronecker_limit_kombi}
For $z\in\mathbb{H}$ with $z\neq\gamma e_{j}$ for any $\gamma\in\Gamma$, 
the elliptic Eisenstein series $\mathcal{E}^{\mathrm{ell}}_{e_{j}}(z,s)$ 
associated to $e_j\in E_{\Gamma}$ 
admits a Laurent expansion at $s=0$ of the form
\begin{align*}
&\mathcal{E}^{\mathrm{ell}}_{e_{j}}(z,s)-
\frac{2^{s}\sqrt{\pi}\,\Gamma(s-\frac{1}{2})}{n_{e_j}\,\Gamma(s)}
\sum\limits_{k=1}^{p_{\Gamma}}
\mathcal{E}^{\mathrm{par}}_{p_{k}}(e_{j},1-s)
\,\mathcal{E}^{\mathrm{par}}_{p_{k}}(z,s)=\\
&-C_{e_j}
-\log\bigl(|H_{e_{j}}(z)|\Im(z)^{C_{e_j}}\Im(e_j)^{C_{e_j}}\bigr)\cdot s+\Landau(s^2),
\end{align*}
where we have set (for $w\in\mathbb{H}$)
$$
C_{w}:=\frac{2\pi}{n_{w}\vol_{\hyp}(\mathcal{F}_{\Gamma})};
$$
further, $H_{e_{j}}(z)$ is a holomorphic function, unique up to multiplication with a complex
constant of absolute value $1$, which vanishes if and only if
$z=\gamma e_{j}$ for any $\gamma\in\Gamma$, which fulfills $\big|H_{e_{j}}(z)\big|^{n_{e_j}}=\big|H_{z}(e_{j})\big|^{n_{z}}$
and which satisfies
\begin{align}
H_{e_{j}}(\gamma z)=\varepsilon_{e_{j}}(\gamma)(c z+d)^{2C_{e_j}}H_{e_{j}}( z)
\label{eq_weight_H}
\end{align}
for any $\gamma=\bigl(\begin{smallmatrix}a&b\\c&d\end{smallmatrix}\bigr)\in\Gamma$. 
Here, $\varepsilon_{e_{j}}(\gamma)\in\mathbb{C}$ is a constant
of absolute value $1$ depending on ${e_{j}}$ and $\gamma$ but which is independent of $z$.
\end{theorem}

\begin{proof}
Using the notation \eqref{def_rest_pareis}, i.e.,
\begin{align*}
R_{e_j}(z,s):=
\frac{2^{s}\sqrt{\pi}\,\Gamma(s-\frac{1}{2})}{n_{e_j}\,\Gamma(s)}
\sum\limits_{k=1}^{p_{\Gamma}}
\mathcal{E}^{\mathrm{par}}_{p_{k}}(e_{j},1-s)
\,\mathcal{E}^{\mathrm{par}}_{p_{k}}(z,s),
\end{align*}
we recall that Proposition \ref{prop_laurent} provides the following Laurent expansion at $s=0$ 
\begin{align}
&\mathcal{E}^{\mathrm{ell}}_{e_{j}}(z,s)-R_{e_j}(z,s)=
\sum\limits_{r=0}^{\infty} a_{r,e_j}(z)\cdot s^{r}
\label{eq_expansion1}
\end{align}
with $a_{r,e_j}(z)\in\mathcal{A}(\Gamma\backslash\mathbb{H})$ for $r\in\mathbb{N}$,
and with $a_{0,e_j}(z)=-2\pi/n_{e_j}\vol_{\hyp}(\mathcal{F}_{\Gamma})$
and
$a_{1,e_j}(z)=\mathcal{K}_{e_j}(z)$. 
Further, for $z\in\mathbb{H}$ with $z\neq\gamma e_{j}$ for any $\gamma\in\Gamma$, 
the function $\mathcal{E}^{\mathrm{ell}}_{e_{j}}(z,s+2)$ is holomorphic at
$s=0$ by lemma \ref{4.2} and non-vanishing by the very definition of the series.
Therefore, we have a Laurent expansions at $s=0$ of the form
\begin{align}
\mathcal{E}^{\mathrm{ell}}_{e_{j}}(z,s+2)&=\sum\limits_{r=0}^{\infty} b_{r,e_j}(z) \cdot s^{r}
\label{eq_expansion2}
\end{align}
with $b_{r,e_j}(z)\in\mathcal{A}(\Gamma\backslash\mathbb{H})$ for $r\in\mathbb{N}$.
Substituting the expansions \eqref{eq_expansion1} and \eqref{eq_expansion2} into the differential equation 
\begin{align*}
\bigl(\Delta_{\hyp}-s(1-s)\bigr)\bigl(\mathcal{E}^{\mathrm{ell}}_{e_{j}}(z,s)-R_{e_j}(z,s)\bigr)=
\bigl(\Delta_{\hyp}-s(1-s)\bigr)\mathcal{E}^{\mathrm{ell}}_{e_{j}}(z,s)=
-s^{2}\mathcal{E}^{\mathrm{ell}}_{e_{j}}(z,s+2),
\end{align*}
we derive the identity
\begin{align*}
\sum\limits_{r=0}^{\infty}\Delta_{\hyp} a_{r,e_{j}}(z)\cdot s^{r}=\sum\limits_{r=1}^{\infty}a_{r-1,e_{j}}(z)\cdot s^{r}
-\sum\limits_{r=2}^{\infty}(a_{r-2,e_{j}}(z)+b_{r-2,e_{j}}(z))\cdot s^{r}.
\end{align*}
Comparison of the coefficients now leads to the following recurrence formula
\begin{align*}
\Delta_{\hyp} a_{r,e_{j}}(z)=a_{r-1,e_{j}}(z)-a_{r-2,e_{j}}(z)-b_{r-2,e_{j}}(z),
\end{align*}
where $a_{r,e_{j}}(z)=b_{r,e_{j}}(z)=0$ for $r<0$. In particular, for $r=0$, we recover
$\Delta_{\hyp} a_{0,e_j}(z)=-\Delta_{\hyp}C_{e_j}=0$. Further, for $r=1$, we get
\begin{align}
\Delta_{\hyp} a_{1,e_{j}}(z)=\Delta_{\hyp} \mathcal{K}_{e_{j}}(z)=
a_{0,e_j}(z)=-C_{e_j}.
\label{eq_identity_const}
\end{align}
Since $\Delta_{\hyp}\log(y)=1$, this implies the identity
\begin{align*}
\Delta_{\hyp} \bigl(\mathcal{K}_{e_{j}}(z)+C_{e_j}\log(y)+C_{e_j}\log(\Im(e_j))\bigr)=-C_{e_j}+C_{e_j}=0.
\end{align*}
Hence, for $z\in\mathbb{H}$ with $z\neq\gamma e_{j}$ for any $\gamma\in\Gamma$, 
the function
$$
h_{e_{j}}(z):=\mathcal{K}_{e_{j}}(z)+C_{e_j}\log(\Im(z))+C_{e_j}\log(\Im(e_j))
$$
is a non-constant, real-valued harmonic function, which fulfills $n_{e_{j}}h_{e_{j}}(z)=n_{z}h_{z}(e_{j})$.
Now, for fixed $w\in\mathbb{H}$ and for a constant $A\in\mathbb{C}$ of absolute value $1$,
we consider the function
\begin{align*}
H_{e_j}(z):
&=A\exp\bigl( -h_{e_{j}}(w)\bigr)
\exp\Biggl(\,\int\limits_{z}^{w}2\frac{\partial}{\partial \tilde{z}}h_{e_{j}}(\tilde{z})d\tilde{z}\Biggr).
\end{align*}
Since $h_{e_{j}}(z)=\mathcal{K}_{e_{j}}(z)+C_{e_j}\log(\Im(z))+C_{e_j}\log(\Im(e_j))$ is harmonic,
the function $H_{e_j}(z)$ is independent of the path from $z$ to $w$ and is analytic in $z$.
In particular, for any path from $z$ to $w$, we have the identity
\begin{align*}
&\Re\Biggl(\,\int\limits_{z}^{w}2\frac{\partial}{\partial \tilde{z}}h_{e_{j}}(\tilde{z})d\tilde{z}\Biggr)=
h_{e_{j}}(w)-h_{e_{j}}(z).
\end{align*}
Therefore, we conclude that
\begin{align*}
\big|H_{e_j}(z)\big|&
=\exp\bigl(-\Re( h_{e_{j}}(w))\bigr)\exp\bigl(h_{e_{j}}(w)-h_{e_{j}}(z)\bigr)
=\exp\bigl(-h_{e_{j}}(z)\bigr).
\end{align*}
Further, $H_{e_j}(z)$ vanishes if and only if $z=\gamma e_{j}$ for some $\gamma\in\Gamma$,
and by the estimate \eqref{eq_estimate_kron} we derive the
following Laurent expansion at $z=\gamma e_{j}$ ($\gamma\in\Gamma$)
\begin{align}
\label{eq_exp_hej}
H_{e_j}(z)=\sum\limits_{r=1}^{\infty}a_{r,\gamma e_j}\cdot(z-\gamma e_{j})^{r}
\end{align}
with $a_{r,\gamma e_j}\in\mathbb{C}$ for $r\geq 1$ and $a_{1,\gamma e_j}\not=0$.
Adding up, we have
\begin{align*}
&h_{e_{j}}(z)=\mathcal{K}_{e_{j}}(z)+C_{e_j}\log(\Im(z))+C_{e_j}\log(\Im(e_j))=-\log\big|H_{e_j}(z)\big|\Longleftrightarrow\\
&\mathcal{K}_{e_{j}}(z)=-\log\big|H_{e_j}(z)\Im(z)^{C_{e_j}}\Im(e_j)^{C_{e_j}}\big|.
\end{align*}
Therefore, the $\Gamma$-invariance of the function
$\mathcal{K}_{e_{j}}(z)$ yields the $\Gamma$-invariance of
$\log\big|H_{e_j}(z)\Im(z)^{C_{e_j}}\Im(e_j)^{C_{e_j}}\big|$. This leads to the identity
\begin{align*}
&\big|H_{e_j}(\gamma z)\big|\Im(\gamma z)^{C_{e_j}}\Im(e_j)^{C_{e_j}}=
|H_{e_j}(z)|\Im(z)^{C_{e_j}} \Im(e_j)^{C_{e_j}}\Longleftrightarrow\\
&\big|H_{e_j}(\gamma z)\big|=\big|H_{e_j}(z)\big|\cdot |cz+d|^{2C_{e_j}}\Longleftrightarrow\\
&\big|H_{e_j}(\gamma z)\big|=\big|H_{e_j}(z)\cdot(cz+d)^{2C_{e_j}}\big|
\end{align*}
for any $\gamma=\bigl(\begin{smallmatrix}a&b\\c&d\end{smallmatrix}\bigr)\in\Gamma$.
From this we deduce that the function
\begin{align*}
&f(z):=\frac{H_{e_j}(\gamma z)}{H_{e_j}(z)(cz+d)^{2C_{e_j}}}
\end{align*}
has absolute value $1$. Moreover, if $z\in\mathbb{H}$ tends to $\gamma' e_{j}$
for some $\gamma'\in\Gamma$, the translate $\gamma z\in\mathbb{H}$ tends to $\gamma\gamma' e_{j}$.
Hence, from \eqref{eq_exp_hej} we get the following Laurent expansions at 
$z=\gamma' e_{j}$ ($\gamma'\in\Gamma$)
\begin{align*}
H_{e_j}(z)&=\sum\limits_{r=1}^{\infty}a_{r,\gamma' e_j}\cdot(z-\gamma' e_{j})^{r},\\
H_{e_j}(\gamma z)&
=\sum\limits_{r=1}^{\infty}a_{r,\gamma\gamma' e_j}\cdot(\gamma z-\gamma\gamma' e_{j})^{r}
=\sum\limits_{r=1}^{\infty}\frac{a_{r,\gamma\gamma' e_j}}{(c\gamma' e_{j}+d)^r}\cdot\frac{(z-\gamma' e_{j})^{r}}{(cz+d)^{r}}
\end{align*}
with $a_{r,\gamma' e_j},a_{r,\gamma\gamma' e_j}\in\mathbb{C}$ for $r\geq 1$,
$a_{1,\gamma' e_j}\not=0$, and $a_{1,\gamma\gamma' e_j}\not=0$; here, we used that for $z,z'\in\mathbb{H}$ we have
the identity
\begin{align*}
\gamma z-\gamma z'=\frac{z-z'}{(cz+d)(cz'+d)}.
\end{align*}
Since $(cz+d)^{2C_{e_j}}$ never vanishes
for $z\in\mathbb{H}$, this implies that the function $f(z)$ is regular for $z\in\mathbb{H}$.
Therefore, by the maximum principle, we obtain that
$f(z)=\varepsilon$ for a constant $\varepsilon=\varepsilon_{e_{j}}(\gamma)\in\mathbb{C}$
of absolute value $1$ depending on ${e_{j}}$ and $\gamma$, but which is independent of $z$.
Hence, we get
\begin{align*}
H_{e_j}(\gamma z)=\varepsilon_{e_{j}}(\gamma)(cz+d)^{2C_{e_j}}H_{e_j}(z),
\end{align*}
as asserted. This completes the proof of the theorem.
\end{proof}

\section{The case of the full modular group}

In this section, we consider the special case that $\Gamma=\PSL_{2}(\mathbb{Z})$.
Hence, we have $e_{\Gamma}=2$ and $p_{\Gamma}=1$,
and we can choose $E_{\Gamma}=\{e_1=i,e_2= \rho=\exp(2\pi i/3)\}$
and $P_{\Gamma}=\{p_1=\infty\}$.
The point $i\in\mathbb{H}$ is an elliptic fixed point of order $n_i=2$ with 
scaling matrix $\sigma_{i}=\id$, and 
the point $\rho\in\mathbb{H}$ is an elliptic fixed point of order $n_\rho=3$ 
with scaling matrix
$$\sigma_{\rho}=\frac{1}{\sqrt{2}}\Bigl(\begin{smallmatrix}
               \sqrt[4]{3}&-1/\sqrt[4]{3}\\ 0& 2/\sqrt[4]{3}
            \end{smallmatrix} \Bigr).$$  
Moreover, $\infty$ is a cusp of width $1$ and scaling matrix $\sigma_{\infty}=\id$.
The hyperbolic volume is given by
$$
\vol_{\hyp}(\mathcal{F}_{\Gamma})=\frac{\pi}{3}
$$
and, therefore, we have $C_i=6/n_{i}=3$ and $C_{\rho}=6/n_{\rho}=2$.

For $k=4,6$, let
\begin{align}
E_{k}(z)
=\sum_{\gamma=\bigl(\begin{smallmatrix}a&b\\c&d
\end{smallmatrix}\bigr)
\in\Gamma_{\infty}\backslash\Gamma}(cz+d)^{-k}
=\frac{1}{2}\sum_{\substack{\left(c,d\right)\in\mathbb{Z}^{2}\\
\left(c,d\right)=1}}(cz+d)^{-k}
\label{eq_holeisk}
\end{align}
denote the holomorphic Eisenstein series of weight $k$, which
is a modular form satisfying
\begin{align*}
E_{k}(\gamma z)=(c z+d)^{k}E_{k}(z)
\end{align*}
for any $\gamma=\bigl(\begin{smallmatrix}a&b\\c&d\end{smallmatrix}\bigr)\in\Gamma$. 
The function $E_{k}(z)$ ($k=4,6$)
is normalized such that we have the Fourier expansions
\begin{align}
E_{4}(z)&=1+240\sum\limits_{m=1}^{\infty}\sigma_{3}(m)e(mz),\label{eq_fourexp_E4}\\
E_{6}(z)&=1+504\sum\limits_{m=1}^{\infty}\sigma_{5}(m)e(mz),\label{eq_fourexp_E6}
\end{align}
respectively, where $\sigma_{k-1}(m)$ ($k=4,6$) denotes the divisor function.
By $\Delta(z)$, we denote the Dedekind's Delta function 
\begin{align}
\Delta(z)=\frac{1}{1728}\bigl(E_4(z)^3-E_6(z)^2\bigr),
\label{eq_delta12}
\end{align}
which is a cusp form of weight $12$ satisfying
\begin{align*}
\Delta(\gamma z)=(c z+d)^{12}\Delta(z)
\end{align*}
for any $\gamma=\bigl(\begin{smallmatrix}a&b\\c&d\end{smallmatrix}\bigr)\in\Gamma$.

Further, let $\mathcal{E}^{\mathrm{par}}_{\infty}(z,s)$ denote the 
parabolic Eisenstein series 
\begin{align*}
\mathcal{E}^{\mathrm{par}}_{\infty}(z,s)
=\sum_{\gamma\in\Gamma_{\infty}\backslash\Gamma}\Im(\gamma z)^{s}
=\frac{1}{2}\sum_{\substack{\left(c,d\right)\in\mathbb{Z}^{2}\\
\left(c,d\right)=1}}\frac{\Im(z)^s}{\abs{cz+d}^{2s}}\,.
\end{align*}
Its parabolic Fourier expansion (with respect to the cusp $\infty$) is given by
\begin{align}
\label{fourier-expansion-Epar}
\mathcal{E}^{\mathrm{par}}_{\infty}(z,s)
=y^{s}+\varphi(s)y^{1-s}+
\sum\limits_{\substack{ m\in\mathbb{Z}\\
m\not=0}     }
\varphi_{m}(s)y^
{1/2}K_{s-1/2}(2\pi|m|y)e(mx),
\end{align}
where $K_{\nu}(\cdot)$ denotes the modified Bessel 
function of the second kind and where
\begin{align*}
\varphi(s)&=\frac{\sqrt{\pi}\,\Gamma(s-\frac{1}{2})}{\Gamma(s)}\,\frac{\zeta(2s-1)}{\zeta(2s)}=
\frac{\Lambda(2s-1)}{\Lambda(2s)},\\
\varphi_{m}(s)&=\frac{2\pi^{s}|m|^{s-1/2}}{\Gamma(s)\zeta(2s)}\sum_{d|m}d^{-2s+1}=\frac{2}
{\Lambda(2s)}\sum_{ab=|m|}\Bigl(\frac{a}{b}\Bigr)^{s-1/2};
\end{align*}
here, we have set $\Lambda(s):=\pi^{-s/2}\Gamma(s/2)\zeta(s)$
with the Riemann zeta function $\zeta(s)$.

We first recall the classical Kronecker limit formula for the parabolic 
Eisenstein series $\mathcal{E}^{\mathrm{par}}_{\infty}(z,s)$ (see, e.g., \cite{Siegel:1980rs} or \cite{Zagier:1992ke}).
\begin{proposition}\label{prop_classical_kl}
For $z\in\mathbb{H}$, the parabolic Eisenstein series $\mathcal{E}^{\mathrm{par}}_{\infty}(z,s)$
admits a Laurent expansion at $s=1$ of the form
\begin{align*}
\mathcal{E}^{\mathrm{par}}_{\infty}(z,s)=
\frac{\vol_{\hyp}(\mathcal{F}_{\Gamma})^{-1}}{s-1}
-\frac{1}{2\pi}\log\bigl(|\Delta(z)|\Im(z)^{6}\bigr)+C+\Landau(s-1)
\end{align*}
with $C=\bigl(6-72\,\zeta'(-1)-6\log(4\pi)\bigr)/\pi$.
At $s=0$, it admits a Laurent expansion of the form
\begin{align*}
\mathcal{E}^{\mathrm{par}}_{\infty}(z,s)=
1+\log\bigl(|\Delta(z)|^{1/6}\Im(z)\bigr)
\cdot s+\Landau(s^2).
\end{align*}
\end{proposition}

An analogous result holds for the elliptic Eisenstein series
 $\mathcal{E}^{\mathrm{ell}}_{e_j}(z,s)$ ($j=1,2$). To prove it, we first
 recall the parabolic Fourier expansion of the elliptic Eisenstein series 
 (see \cite{Pippich06-Ref} for the special case 
 $\Gamma=\PSL(\mathbb{Z})$, or \cite{Pippich-Preprint-In-Preparation-01} for an arbitrary Fuchsian subgroup $\Gamma$
 of the first kind).
 For $z\in\mathbb{H}$ with $\Im(z)>\Im(\gamma e_{j})$ for any 
$\gamma\in\Gamma$, the elliptic Eisenstein 
series $\mathcal{E}^{\mathrm{ell}}_{e_j}(z,s)$ admits the parabolic Fourier expansion
(with respect to the cusp $\infty$) 
\begin{align}
\mathcal{E}^{\mathrm{ell}}_{e_j}(z,s)=
\sum_{m\in\mathbb{Z}}a_{m;\infty,e_j}(y,s)e(mx)
\label{exp_ell_par}
\end{align}
with coefficients given by
\begin{align*}
a_{0;\infty,e_j}(y,s)
&=\frac{2^{s}\sqrt{\pi}\,\Gamma\bigl(s-\frac{1}{2}\bigr)}{n_{e_j}\Gamma(s)}\,\sum_{k=0}^{\infty}\frac{(s-\frac{1}{2})_{k}(\frac{s}{2})_{k}}
{k!\,(\frac{s}{2}+\frac{1}{2})_{k}}\, y^{1-s-2k}
\,\mathcal{E}^{\mathrm{par}}_{\infty}(e_j,s+2k),\\
a_{m;\infty,e_j}(y,s)&=\frac{2^{s}y^{s}}{n_{e_{j}}}
\sum_{k_{1}=0}^{\infty}\sum_{k_{2}=0}^{\infty}
\frac{(\frac{s}{2})_{k_{1}}\,(\frac{s}{2})_{k_{2}}}
{k_{1}!\, k_{2}!}
\,I_{m}(y,s;k_{1},k_{2})\, V^{\mathrm{par}}_{\infty,m}(e_j,s+2k_{1}+2k_{2}) \qquad(m\not=0),
\end{align*}
with 
\begin{align*}
I_{m}(y,s;k_{1},k_{2})&=\int\limits_{-\infty}^{\infty}(y+it)^{-s-2k_{1}}(y-it)^{-s-2k_{2}}e(-mt)\,dt,\\
V^{\mathrm{par}}_{\infty,m}(z,s)
&=\sum_{\gamma\in\Gamma_{\infty}\backslash\Gamma}\Im(\gamma z)^{s}
e\bigl(-m\Re(\gamma z)\bigr).
\end{align*}

\begin{proposition}\label{prop_kron_special}
For $z\in\mathbb{H}$  with $z\neq\gamma i$ for any $\gamma\in\Gamma$, 
the elliptic Eisenstein series $\mathcal{E}^{\mathrm{ell}}_{i}(z,s)$
admits a Laurent expansion at $s=0$ of the form
\begin{align*}
&\mathcal{E}^{\mathrm{ell}}_{i}(z,s)-\frac{2^{s-1}\sqrt{\pi}\,\Gamma(s-\frac{1}{2})}{\Gamma(s)}\,
\mathcal{E}^{\mathrm{par}}_{\infty}(i,1-s)\,\mathcal{E}^{\mathrm{par}}_{\infty}(z,s)
=\\
&-3+\bigl(-\log\bigl(|E_{6}(z)|\Im(z)^{3}\bigr)+B_{i}\bigr)
\cdot s+\Landau(s^2)
\end{align*}
with 
$B_{i}=-72\,\zeta'(-1)+3\log(2\pi)-12\log\bigl(\Gamma(\frac{1}{4})\bigr)$.\\
Further, for $z\in\mathbb{H}$  with $z\neq\gamma \rho$ 
for any $\gamma\in\Gamma$, the elliptic Eisenstein series 
$\mathcal{E}^{\mathrm{ell}}_{\rho}(z,s)$ admits a Laurent expansion 
at $s=0$ of the form
\begin{align*}
&\mathcal{E}^{\mathrm{ell}}_{\rho}(z,s)-\frac{2^{s}\sqrt{\pi}\,\Gamma(s-\frac{1}{2})}{3\,\Gamma(s)}\,
\mathcal{E}^{\mathrm{par}}_{\infty}(\rho,1-s)\,\mathcal{E}^{\mathrm{par}}_{\infty}(z,s)
=\\
&-2+\Bigl(-\log\bigl(|E_{4}(z)|\Im(z)^{2}\Im(\rho)^{2}\bigr)+B_{\rho}\Bigr)\cdot s+\Landau(s^2)
\end{align*}
with
$B_{\rho}=-48\,\zeta'(-1)+4\log(\frac{2\pi}{\sqrt{3}})-12\log\bigl(\Gamma(\frac{1}{3})\bigr)$.
\end{proposition}

\begin{proof}
By Theorem \ref{theo_kronecker_limit_kombi}, for
$z\in\mathbb{H}$  with $z\neq\gamma e_j$ for any $\gamma\in\Gamma$ ($j=1,2$), 
we have a Laurent expansion at $s=0$ of the form
\begin{align*}
\mathcal{E}^{\mathrm{ell}}_{e_{j}}(z,s)-
\frac{2^{s}\sqrt{\pi}\,\Gamma(s-\frac{1}{2})}{n_{e_j}\,\Gamma(s)}\,
\mathcal{E}^{\mathrm{par}}_{\infty}(e_{j},1-s)
\,\mathcal{E}^{\mathrm{par}}_{\infty}(z,s)=
-C_{e_j}
+
\mathcal{K}_{e_j}(z)\cdot s+\Landau(s^2)
\end{align*}
with
\begin{align}
\mathcal{K}_{e_j}(z)=
-\log\bigl(|H_{e_{j}}(z)|\Im(z)^{C_{e_j}}\Im(e_j)^{C_{e_j}}\bigr),
\label{eq_KH}
\end{align}
for a holomorphic function $H_{e_{j}}(z)$ with properties given in Theorem \ref{theo_kronecker_limit_kombi}.

To explicitly determine the function $H_{e_{j}}(z)$ ($j=1,2$), we first determine its behaviour as $y\to\infty$.
To do this, we let 
$z\in\mathbb{H}$ be such that $\Im(z)>\Im(\gamma e_j)$ for any $\gamma\in\Gamma$,
and we consider the parabolic Fourier expansion of $\mathcal{K}_{e_j}(z)$, which 
is of the form
\begin{align*}
\mathcal{K}_{e_j}(z)=
\sum\limits_{m\in\mathbb{Z}}b_{m;e_j}(y)e(mx)
\end{align*}
with coefficients given by
$$
b_{m;e_j}(y)=\int\limits_{0}^{1}\mathcal{K}_{e_j}(z)e(-mx).
$$
Since $\mathcal{K}_{e_j}(z)$ is real-valued, we have $b_{-m;e_j}(y)=\overline{b}_{m;e_j}(y)$. 
Further, from the differential equation (see \eqref{eq_identity_const})
$$
\Delta_{\hyp}\mathcal{K}_{e_j}(z)=-C_{e_j}=-\frac{2\pi}{n_{e_j}\vol_{\hyp}(\mathcal{F}_{\Gamma})},
$$
we obtain for $m=0$ the identity
$
\Delta_{\hyp}b_{0;e_j}(y)=-C_{e_{j}},
$ 
and, for $m\not=0$, the differential equation
$
b''_{m;e_j}(y)=(2\pi m)^2\,b_{m;e_j}(y).
$
From this we derive that
\begin{align*}
b_{0;e_j}(y)&=-C_{e_{j}}\log(y)+A_{e_j}y+B_{e_j},\\
b_{m;e_j}(y)&=A_{m;e_j}\exp(-2\pi m y)+A'_{m;e_j}\exp(2\pi m y) \qquad(m\not=0)
\end{align*}
with constants $A_{e_j}, B_{e_j}\in\mathbb{R}$, 
and with constants $A_{m;e_j}, A'_{m;e_j}\in\mathbb{C}$
satisfying $A'_{-m;e_j}=\overline{A}_{m;e_j}$. 
However, from the parabolic Fourier expansion of the function 
\begin{align*}
&\mathcal{E}^{\mathrm{ell}}_{e_{j}}(z,s)-
\frac{2^{s}\sqrt{\pi}\,\Gamma\bigl(s-\frac{1}{2}\bigr)}{n_{e_j}\Gamma(s)}\,
\mathcal{E}^{\mathrm{par}}_{\infty}(e_{j},1-s)
\,\mathcal{E}^{\mathrm{par}}_{\infty}(z,s),
\end{align*}
which is obtained by combining \eqref{exp_ell_par} with \eqref{fourier-expansion-Epar}, we 
conclude that $A'_{m;e_j}=0$ for $m>0$ and $A_{m;e_j}=0$ for $m<0$.
Hence, we can write
\begin{align}
\mathcal{K}_{e_j}(z)=-C_{e_{j}}\log(y)+
A_{e_j}y+B_{e_j}+
\sum\limits_{m=1}^{\infty}A_{m;e_j} e(mz)+
\sum\limits_{m=1}^{\infty}\overline{A}_{m;e_j}e(-m\overline{z}).
\label{eq_expansK}
\end{align}
To determine the constants $A_{e_j}, B_{e_j}\in\mathbb{R}$, we introduce
the notation
\begin{align*}
h(s):=\frac{2^{s}\sqrt{\pi}\,\Gamma\bigl(s-\frac{1}{2}\bigr)}{n_{e_j}\Gamma(s)},
\end{align*} 
and we consider the constant term $\widetilde{a}_{0;\infty,e_j}(y,s)$ of the parabolic Fourier 
expansion of the function
\begin{align*}
\mathcal{E}^{\mathrm{ell}}_{e_{j}}(z,s)-
h(s)\,
\mathcal{E}^{\mathrm{par}}_{\infty}(e_{j},1-s)
\,\mathcal{E}^{\mathrm{par}}_{\infty}(z,s)=
\mathcal{E}^{\mathrm{ell}}_{e_{j}}(z,s)-
h(s)\,
\mathcal{E}^{\mathrm{par}}_{\infty}(e_{j},s)
\,\mathcal{E}^{\mathrm{par}}_{\infty}(z,1-s).
\end{align*}
The constant term $\widetilde{a}_{0;\infty,e_j}(y,s)$ is given by 
\begin{align*}
\widetilde{a}_{0;\infty,e_j}(y,s)&=a_{0;\infty,e_j}(y,s)-
h(s)\,
\mathcal{E}^{\mathrm{par}}_{\infty}(e_{j},s)
\,\bigl(y^{1-s}+\varphi(1-s)y^{s}\bigr)\\
&=F_{e_j}(y,s)-h(s)\,\varphi(1-s)y^{s}\,\mathcal{E}^{\mathrm{par}}_{\infty}(e_{j},s).
\end{align*}
with
\begin{align*}
F_{e_j}(y,s)&:=
h(s)\,\sum_{k=1}^{\infty}\frac{(s-\frac{1}{2})_{k}(\frac{s}{2})_{k}}
{k!\,(\frac{s}{2}+\frac{1}{2})_{k}}\, y^{1-s-2k}
\, \mathcal{E}^{\mathrm{par}}_{\infty}(e_{j},s+2k).
\end{align*}
Since the function $\mathcal{E}^{\mathrm{par}}_{\infty}(e_{j},s+2k)$ ($k\in\mathbb{N}$, $k>0$)
is holomorphic and non-vanishing for $s\in\mathbb{C}$ with $\Re(s)>-1$, and using
the Laurent expansion at $s=0$ 
\begin{align*}
h(s)\,\frac{(s-\frac{1}{2})_{k}(\frac{s}{2})_{k}}
{k!\,(\frac{s}{2}+\frac{1}{2})_{k}}\,y^{1-s-2k}\,
=-\frac{\pi y^{1-2k}}{n_{e_j}(k-2k^2)}\cdot s^2 +\Landau(s^3),
\end{align*}
we derive that $F_{e_j}(y,s)=\Landau(s^2)$ at $s=0$.
Further, we have at $s=0$ the following Laurent expansions 
\begin{align*}
-h(s)\,\varphi(1-s)y^{s}&=-C_{e_j}-C_{e_j}\bigl(24\,\zeta'(-1)+\log(8\pi^2)+\log(y)\bigr)
\cdot s+\Landau(s^2),\\
\mathcal{E}^{\mathrm{par}}_{\infty}(e_{j},s)&=
1+\log\bigl(|\Delta(e_{j})|^{1/6}\Im(e_{j})\bigr)
\cdot s+\Landau(s^2),
\end{align*}
where the last expansion follows from proposition \ref{prop_classical_kl}.
These expansions lead to the following Laurent expansion at $s=0$
\begin{align*}
\widetilde{a}_{0;\infty,e_j}(y,s)=
-C_{e_j}- C_{e_j}
\Bigl(24\,\zeta'(-1)+\log(8\pi^2)+\log(y)+\log\bigl(|\Delta(e_{j})|^{1/6}\Im(e_{j})\bigr)\Bigr)
 \cdot s+\Landau(s^2).
\end{align*}
From this we derive that
\begin{align*}
b_{0;e_j}(y)
&=-C_{e_{j}}\log(y)-C_{e_{j}}\Bigl(24\,\zeta'(-1)+\log(8\pi^2)+\log\bigl(|\Delta(e_{j})|^{1/6}\Im(e_{j})\bigr)\Bigr),
\end{align*}
and, therefore, we get
\begin{align}
A_{e_j}&=0,\notag\\
B_{e_j}
&=-C_{e_{j}}\bigl(24\,\zeta'(-1)+\log(8\pi^2)+\log\bigl(|\Delta(e_{j})|^{1/6}\bigr)\bigr)
-C_{e_{j}}\log(\Im(e_{j})).
\label{eq_formulaB}
\end{align}
Introducing the notation
$$
f_{e_j}(z)=\exp\biggl( -2\sum\limits_{m=1}^{\infty}A_{m;e_j} e(mz) \biggr),
$$
we derive from \eqref{eq_expansK} the equality
\begin{align*}
\mathcal{K}_{e_j}(z)&=-C_{e_{j}}\log(y)+B_{e_j}
-\log\bigl(f_{e_j}(z)^{1/2}\bigr)
-\log\bigl(\overline{f}_{e_j}(z)^{1/2}\bigr)\\
&=B_{e_j}
-\log\bigl(|f_{e_{j}}(z)|\Im(z)^{C_{e_j}}\bigr).
\end{align*}
From this and \eqref{eq_KH}, we derive the identity 
$$
\mathcal{K}_{e_j}(z)=
-\log\bigl(|H_{e_{j}}(z)|\Im(z)^{C_{e_j}}\Im(e_j)^{C_{e_j}}\bigr)
$$
\begin{align*}
\log\bigl(|H_{e_{j}}(z)|\Im(z)^{C_{e_j}}\Im(e_j)^{C_{e_j}}\bigr)
&=-B_{e_j}
+\log\bigl(|f_{e_{j}}(z)|\Im(z)^{C_{e_j}}\bigr),\text{ i.e.}\\
|H_{e_{j}}(z)|
&=\Im(e_j)^{-C_{e_j}}\exp(-B_{e_j})
|f_{e_{j}}(z)|.
\end{align*}
Since $|f_{e_j}(z)|\to 1$ as $y\to\infty$, we have
$|H_{e_j}(z)|\to \Im(e_j)^{-C_{e_j}}\exp(-B_{e_j})$ as $y\to\infty$,
which implies 
\begin{align}
H_{e_j}(z)=A\,  \Im(e_j)^{-C_{e_j}}\exp(-B_{e_j})+\Landau(\exp(-2\pi y))
\label{eq_bound_Hej}
\end{align}
as $y\to\infty$ with a constant $A\in\mathbb{C}$ (depending on $e_j$)
of absolute value $1$.
We now consider the function
$$
f_{e_j}(z):=\frac{H_{e_j}(z)}{E_{2C_{e_j}}(z)}
$$
with the Eisenstein series $E_{2C_{e_j}}(z)$ of weight $2C_{e_j}$. 
The Eisenstein series $E_{2C_{e_j}}(z)$ vanishes if and only if 
$z=\gamma e_{j}$ for any $\gamma\in\Gamma$, and it admits a
simple zero at $z=\gamma e_{j}$ for any $\gamma\in\Gamma$.
Similarly, the function $H_{e_j}(z)$ vanishes if and only if 
$z=\gamma e_{j}$ for any $\gamma\in\Gamma$, and, by \eqref{eq_exp_hej},
it admits a simple zero at $z=\gamma e_{j}$ for any $\gamma\in\Gamma$.
Therefore, the function $f_{e_j}(z)$ 
is a regular function on $\mathbb{H}$. Moreover, from the 
asymptotics
\begin{align}
E_{2C_{e_j}}(z)&=1+\Landau(\exp(-2\pi y))\label{eq_bound_E}
\end{align}
as $y\to\infty$, which can be deduced from the Fourier expansion 
\eqref{eq_fourexp_E4}, \eqref{eq_fourexp_E6}, respectively, 
and from the bound \eqref{eq_bound_Hej}, we deduce that 
$f_{e_j}(z)$ is bounded as $y\to\infty$. Hence, 
the function $f_{e_j}(z)$ is a modular function
with a finite character. Therefore, there exists $m_{e_j}\in\mathbb{N}$, $m_{e_j}\geq 1$,
such that $f_{e_j}(z)^{m_{e_j}}$ is a modular function with trivial character.
Hence, we get $f_{e_j}(z)^{m_{e_j}}=c_{e_j}$ for a constant $c_{e_j}\in\mathbb{C}$
and, therefore, we get the equality
\begin{align}
H_{e_j}(z)^{m_{e_j}}=c_{e_j}E_{2C_{e_j}}(z)^{m_{e_j}}.
\label{eq_hmiteis}
\end{align}
From the asymptotics $\eqref{eq_bound_Hej}$, we derive
\begin{align*}
H_{e_j}(z)^{m_{e_j}}=A^{m_{e_j}}  \Im(e_j)^{-m_{e_j}C_{e_j}} \exp(-m_{e_j} B_{e_j})+\Landau(\exp(-2\pi y))
\end{align*}
as $y\to\infty$, which, together with the bound \eqref{eq_bound_E},
leads to the equality
\begin{align}
c_{e_j}=A^{m_{e_j}}   \Im(e_j)^{-m_{e_j}C_{e_j}}  \exp(-m_{e_j} B_{e_j}).
\label{eq_const7}
\end{align}
Substituting \eqref{eq_const7} into \eqref{eq_hmiteis}, we obtain the identity
\begin{align*}
\log\bigl(|H_{e_j}(z)^{m_{e_j}}|\bigr)=\log\bigl(  \Im(e_j)^{-m_{e_j}C_{e_j}}  \exp(-m_{e_j} B_{e_j})|E_{2C_{e_j}}(z)^{m_{e_j}}|\bigr),
\end{align*}
from which we deduce
\begin{align*}
\log\bigl(|H_{e_j}(z)|\bigr)=-B_{e_j}+\log\bigl(  \Im(e_j)^{-C_{e_j}}|E_{2C_{e_j}}(z)|\bigr).
\end{align*}
Therefore, we get
\begin{align*}
\mathcal{K}_{e_j}(z)=
-\log\bigl(|H_{e_{j}}(z)|\Im(z)^{C_{e_j}} \Im(e_j)^{C_{e_j}}\bigr)
=-\log\bigl(|E_{2C_{e_j}}(z)|\Im(z)^{C_{e_j}}\bigr)+B_{e_j}.
\end{align*}
Finally, from the well-known formulas (see, e.g., \cite{Diamond:2005uf}, p.~7)
\begin{align*}
E_4(i)=\frac{3\,\Gamma(\frac{1}{4})^8}{(2\pi)^6} \quad\text{ resp. }\quad
E_6(\rho)=\frac{2^33^3\,\Gamma(\frac{1}{3})^{18}}{(2\pi)^{12}},
\end{align*}
we derive
\begin{align*}
|\Delta(i)|^{1/6}=\frac{E_4(i)^{1/2}}{1728^{1/6}}=
\frac{\Gamma(\frac{1}{4})^4}{2\,(2\pi)^3}, \quad
|\Delta(\rho)|^{1/6}
=\frac{
E_6(\rho)^{1/3}}{
1728^{1/6}}=
\frac{
3^{1/2}
\,\Gamma(\frac{1}{3})^{6}}{(2\pi)^{4}},
\end{align*}
respectively.
Substituting these identities for $j=1,2$ into \eqref{eq_formulaB},
$$
B_{e_j}
=-C_{e_{j}}\bigl(24\,\zeta'(-1)+\log(8\pi^2)+\log\bigl(|\Delta(e_{j})|^{1/6}\bigr)\bigr)
-C_{e_{j}}\log(\Im(e_{j}))
$$
we obtain the formulas
\begin{align*}
B_{i}
&=-3\Bigl(24\,\zeta'(-1)-\log(2\pi)+4\log\bigl(\Gamma(\textstyle{\frac{1}{4}})\bigr)\Bigr),
\\
B_{\rho}
&=-2\Bigl(24\,\zetaÕ(-1)-2\log(2\pi)+2\log(\sqrt{3})+6\log\bigl(\Gamma(\textstyle{\frac{1}{3}})\bigr)\Bigr),
\end{align*}
observing that $\Im(\rho)=\sqrt{3}/2$.
This completes the proof of the proposition.
\end{proof}

Combining the classical Kronecker limit formula given in Proposition \ref{prop_classical_kl} 
with the Kronecker limit type formula for the elliptic Eisenstein series given in Proposition \ref{prop_kron_special},
we deduce the following Kronecker limit formula for the elliptic Eisenstein series for $\PSL_2(\mathbb{Z})$.

\begin{corollary}\label{cor_klf1}
For $z\in\mathbb{H}$  with $z\neq\gamma i$ for any $\gamma\in\Gamma$, 
there is
a Laurent expansion at $s=0$ of the form
\begin{align*}
\mathcal{E}^{\mathrm{ell}}_{i}(z,s)
=-\log\bigl(|E_{6}(z)|\,|\Delta(z)|^{-1/2}\bigr)
\cdot s+\Landau(s^2).
\end{align*}
Further, for $z\in\mathbb{H}$  with $z\neq\gamma \rho$ 
for any $\gamma\in\Gamma$, 
we have a Laurent expansion 
at $s=0$ of the form
\begin{align*}
\mathcal{E}^{\mathrm{ell}}_{\rho}(z,s)
=-\log\bigl(|E_{4}(z)|\,|\Delta(z)|^{-1/3}\bigr)\cdot s+\Landau(s^2).
\end{align*}
\end{corollary}

\begin{proof}
Using the notation
\begin{align*}
h(s):=\frac{2^{s}\sqrt{\pi}\,\Gamma\bigl(s-\frac{1}{2}\bigr)}{n_{e_j}\Gamma(s)},
\end{align*} 
Proposition \ref{prop_classical_kl} yields the following Laurent expansions at $s=0$
\begin{align*}
h(s)\mathcal{E}^{\mathrm{par}}_{\infty}(z,1-s)&=
C_{e_j}+
C_{e_j}\bigl( C'+
\log\bigl(|\Delta(z)|^{1/6}\Im(z)\bigr)
\bigr)\cdot s+\Landau(s^2),\\
\mathcal{E}^{\mathrm{par}}_{\infty}(e_{j},s)&=
1+\log\bigl(|\Delta(e_{j})|^{1/6}\Im(e_j)\bigr)\cdot s+\Landau(s^2)
\end{align*}
with $C'=24\,\zeta'(-1)+\log(8\pi^2)$.
Therefore, we obtain a Laurent expansion at $s=0$ of the form
\begin{align*}
&h(s)\,\mathcal{E}^{\mathrm{par}}_{\infty}(e_{j},1-s)
\,\mathcal{E}^{\mathrm{par}}_{\infty}(z,s)=
h(s)\,\mathcal{E}^{\mathrm{par}}_{\infty}(e_{j},s)
\,\mathcal{E}^{\mathrm{par}}_{\infty}(z,1-s)=\\
&C_{e_j}+C_{e_j}\bigl( C'+
\log\bigl(|\Delta(z)|^{1/6}|\Delta(e_{j})|^{1/6}\Im(z)\Im(e_j)\bigr)
\bigr)\cdot s +\Landau(s^2).
\end{align*}
Substituting this Laurent expansion into the Laurent expansion given by
Proposition \ref{prop_kron_special}, using formula \eqref{eq_formulaB}, and adding up, 
we conclude that for $z\in\mathbb{H}$  with 
$z\neq\gamma e_j$ for any $\gamma\in\Gamma$, the elliptic Eisenstein series 
$\mathcal{E}^{\mathrm{ell}}_{e_j}(z,s)$ ($j=1,2$) admits a Laurent expansion 
at $s=0$ of the form
\begin{align*}
\mathcal{E}^{\mathrm{ell}}_{e_j}(z,s)
&=\Bigl(-\log\bigl(|E_{2 C_{e_j}}(z)|\Im(z)^{C_{e_j}}\bigr)+
C_{e_j}\log\bigl(|\Delta(z)|^{1/6}\Im(z)\bigr)
\Bigr)\cdot s+\Landau(s^2)\\
&=-\log\bigl(|E_{2 C_{e_j}}(z)|\,|\Delta(z)|^{-C_{e_j}/6}\bigr)\cdot s+\Landau(s^2).
\end{align*}
This completes the proof of the corollary.
\end{proof}

\begin{corollary}\label{cor_klf2}
For $z\in\mathbb{H}$  with $z\neq\gamma i$ for any $\gamma\in\Gamma$, 
there is
a Laurent expansion at $s=0$ of the form
\begin{align*}
\mathcal{E}^{\mathrm{ell}}_{i}(z,s)
=-\log\bigl(|j(i)-j(z)|^2\bigr)
\cdot s+\Landau(s^2).
\end{align*}
Further, for $z\in\mathbb{H}$  with $z\neq\gamma \rho$ 
for any $\gamma\in\Gamma$, 
we have a Laurent expansion 
at $s=0$ of the form
\begin{align*}
\mathcal{E}^{\mathrm{ell}}_{\rho}(z,s)
=-\log\bigl(|j(\rho)-j(z)|^3\bigr)\cdot s+\Landau(s^2).
\end{align*}
\end{corollary}

\begin{proof}
By means of the identities
$$
|j(i)-j(z)|^2=|E_{6}(z)|\,|\Delta(z)|^{-1/2}, \quad |j(\rho)-j(z)|^3=|E_{4}(z)|\,|\Delta(z)|^{-1/3},
$$
Corollary \ref{cor_klf2} is an immediate consequence of Corollary \ref{cor_klf1}.
\end{proof}

\begin{remark}
The proof of Proposition \ref{prop_laurent} provides an expression of $|j(e_j)-j(z)|$
in terms of spectral data, more precisely, in terms of the functions $F_{e_j}(z)$ resp.
$G_{e_j}(z)$, given by \eqref{eq_def_F}, \eqref{eq_def_G}, respectively.
\end{remark}


\begin{remark}
Let $\Gamma$ be a Fuchsian subgroup satisfying $p_{\Gamma}=1$. Then, at $s=1$, we have
\begin{align*}
\mathcal{E}^{\mathrm{par}}_{p_1}(z,s)&=
\frac{\vol_{\hyp}(\mathcal{F}_{\Gamma})^{-1}}{s-1}
+\mathcal{K}^{\mathrm{par}}(z)+\Landau(s-1),\\
\varphi_{p_1,p_1}(s)&=
\frac{\vol_{\hyp}(\mathcal{F}_{\Gamma})^{-1}}{s-1}
+\kappa_{p_1,p_1}+\Landau(s-1).
\end{align*}
with the Kronecker limit function $\mathcal{K}^{\mathrm{par}}(z)$ (see \cite{Jorgenson:2005nx})
and the scattering constant $\kappa_{p_1,p_1}$.
Since $p_{\Gamma}=1$, the functional equations imply that, at $s=0$, we have
\begin{align*}
\mathcal{E}^{\mathrm{par}}_{p_1}(z,s)=
1-\vol_{\hyp}(\mathcal{F}_{\Gamma})\mathcal{K}^{\mathrm{par}}(z)-\kappa_{p_1,p_1}+\Landau(s^2).
\end{align*}
From this, it seems to be possible to determine explicitly determine the Laurent
expansion of the function $R_{e_j}(z,s)$ given by \eqref{def_rest_pareis}, and then to
establish a generalization of Corollary \ref{cor_klf2}.
In case that $p_{\Gamma}>1$, the study of the function $R_{e_j}(z,s)$ is more difficult, 
since the behaviour of $\mathcal{E}^{\mathrm{par}}_{p_j}(z,s)$ at $s=0$ is not known
in general.
\end{remark}

\section{Relation to the automorphic Green's function} 

For $z, w\in M$ with $z\not=w$ and $s\in\mathbb{C}$ with $\mathrm{Re}(s)>1$,
the \textit{automorphic Green's function} on $M$
is defined as
\begin{align*}
G_s(z,w) = \sum_{\gamma \in\Gamma}g_s(z,\gamma w),
\end{align*}
where $g_s(z,w)$ is the Green's function on $\mathbb{H}$ given by
\begin{align*}
g_s(z,w)= \frac{1}{4\pi}\frac{\Gamma(s)^2}{\Gamma(2s)} u(z,w)^{-s}F\Bigl(s,s;2s;-\frac{1}{u(z,w)}\Bigr)
\end{align*}
with $u(z,w)$ defined by \eqref{def_u} and with the hypergeometric function $F(s,s;2s;Z)$  
recalled in subsection \ref{subsection_specfct}. In the literature, there are
different normalizations of $g_s(z,w)$. We follow the definition given in 
\cite{Iwaniec:1997ys}, p. 26, with a minus sign in front of $1/u$ corrected. The Green's function studied in \cite{Fay:1977pq}
and in \cite{Hejhal:1983vn} differs from $G_s(z,w)$ by a minus sign, the Green's function studied in 
\cite{gr84} equals $-4\pi G_s(z,w)$.

For $z\in\mathbb{H}$ with $z\neq\gamma e_{j}$ for any $\gamma\in\Gamma$,
and $s\in\mathbb{C}$, we will consider the following function
\begin{align}\label{def_Gell}
G^{\mathrm{ell}}_{e_j}(z,s)&:= \frac{1}{n_{e_j}}\,G_s(e_{j},z) 
=\sum_{\gamma\in\Gamma_{e_j}\backslash\Gamma}
g_s(i,\sigma_{e_{j}}^{-1}\gamma z)\,.
\end{align}
Referring to \cite{Hejhal:1983vn},  \cite{Iwaniec:2002zr}, or \cite{Kubota:1973uq},
where detailed proofs are provided, we recall that the series \eqref{def_Gell} converges 
absolutely and locally uniformly for any $z\in M$, $z\not=e_j$, and $s\in\mathbb{C}$ with 
$\mathrm{Re}(s)>1$, and that it is holomorphic for $s\in\mathbb{C}$
with $\mathrm{Re}(s)>1$. Moreover, it is invariant with respect to $\Gamma$.
Futhermore, $G^{\mathrm{ell}}_{e_j}(z,s)$ admits a meromorphic continuation
to the whole $s$-plane, assuming $z, w\in M$ with $z\not=w$.


To compare the automorphic Green's function with the elliptic Eisenstein series, 
we first prove the following infinite relation.

\begin{lemma}\label{lemma_rel_Pell_Greens}
For $z\in\mathbb{H}$ with $z\neq\gamma e_{j}$ for any $\gamma\in\Gamma$,
and $s\in\mathbb{C}$ with $\Re(s)>1$, we have the relation
\begin{align}\label{rel_Pell_Vell}
\mathcal{G}^{\mathrm{ell}}_{e_j}(z,s)=\frac{ 2^{s}}{4\pi}\frac{\Gamma(s)^2}{\Gamma(2s)}
\,\sum\limits_{k=0}^{\infty}\frac{(\frac{s}{2})_{k}(\frac{s}{2}+\frac{1}{2})_{k}}{k!(s+\frac{1}{2})_{k}}
P^{\mathrm{ell}}_{e_j}(z,s+2k)
\end{align}
with $P^{\mathrm{ell}}_{e_j}(z,s)$ defined by \eqref{eq_ellpoin_P}.
\end{lemma}
\begin{proof}
The absolute and local uniform convergence of the series in the claimed 
relation for fixed $z\in\mathbb{H}$ with $z\neq\gamma e_{j}$ for any $\gamma\in\Gamma$
and $s\in\mathbb{C}$ with $\Re(s)>1$, can be proven along the lines of the proof
given in Lemma \ref{lemma_rel_Pell_Vell}.

In the next step, for $z\not=w$, we express the function
\begin{align*}
g_s(z,w)= \frac{1}{4\pi}\frac{\Gamma(s)^2}{\Gamma(2s)} u(z,w)^{-s}F\Bigl(s,s;2s;-\frac{1}{u(z,w)}\Bigr);
\end{align*}
in terms of $\cosh(d_{\mathrm{hyp}}(z,w))$, by applying formula 9.134.1 of \cite{Gradshteyn:2007ys}, namely
\begin{align*}
F(\alpha, \beta;2\beta; Z)=\Bigl(1-\frac{Z}{2}\Bigr)^{-\alpha} F\Bigl(\frac{\alpha}{2},\frac{\alpha+1}{2};\beta+\frac{1}{2};\frac{Z^2}{(2-Z)^2}\Bigr).
\end{align*}
Letting $\alpha:=s$, $\beta:=s$, and $Z:=-1/u(z,w)$, we obtain
\begin{align*}
u(z,w)^{-s}F\Bigl(s, s;2s;-\frac{1}{u(z,w)}\Bigr)&=2^{s}\cosh(d_{\mathrm{hyp}}(z,w))^{-s} F\Bigl(\frac{s}{2},\frac{s+1}{2};s+\frac{1}{2};\frac{1}{\cosh(d_{\mathrm{hyp}}(z,w))^2}\Bigr)\\
&=2^{s}\sum\limits_{k=0}^{\infty}\frac{(\frac{s}{2})_{k}(\frac{s}{2}+\frac{1}{2})_{k}}{k!(s+\frac{1}{2})_{k}}
\cosh(d_{\mathrm{hyp}}(z,w))^{-s-2k}.
\end{align*}
Hence, for $z\not=w$, we get the desired expression
\begin{align*}
g_s(z,w)= \frac{ 2^{s}}{4\pi}\frac{\Gamma(s)^2}{\Gamma(2s)}\sum\limits_{k=0}^{\infty}\frac{(\frac{s}{2})_{k}(\frac{s}{2}+\frac{1}{2})_{k}}{k!(s+\frac{1}{2})_{k}}
\cosh(d_{\mathrm{hyp}}(z,w))^{-s-2k}\end{align*}
Applying Definition \eqref{eq_ellpoin_P} and Definition \eqref{def_Gell}, the claimed relation can now be derived by changing the order of summation. This completes the proof of the lemma.
\end{proof}
 
\begin{remark}
Using the relation given in Lemma \ref{lemma_rel_Pell_Greens}, 
one can establish a new proof of the meromorphic continuation
of the automorphic Green's function, along the lines of the proof of
Theorem \ref{theo_mero_1}.
\end{remark}

\begin{proposition}\label{prop_rel_Eell_Gell}
For $z\in\mathbb{H}$ with $z\neq\gamma e_{j}$ for any $\gamma\in\Gamma$,
and $s\in\mathbb{C}$ with $\Re(s)>1$, we have the relation
\begin{align}
\label{rel_Eell_Gell}
\mathcal{E}^{\mathrm{ell}}_{e_j}(z,s)- \frac{2\sqrt{\pi}\,\Gamma(s+\frac{1}{2})}{\Gamma(s)} \,G^{\mathrm{ell}}_{e_j}(z,s)= 
\sum\limits_{k=1}^{\infty}\frac{(\frac{s}{2})_{k} \,a_k(s)}{k!}\,P^{\mathrm{ell}}_{e_j}(z,s+2k)
\end{align}
with $a_k(s):=1-(\frac{s}{2}+\frac{1}{2})_{k}/(s+\frac{1}{2})_{k}$.
\end{proposition}

\begin{proof}
By the duplication formula \ref{eq-duplication-Gamma}, we deduce
\begin{align*}
\frac{2^{s+1}\sqrt{\pi}\,\Gamma(s+\frac{1}{2})}{\Gamma(s)}=\frac{4\pi\,\Gamma(2 s)}{\Gamma(s)^2}.
\end{align*}
Therefore, the assertion can easily be proven by combining Lemma \ref{lemma_rel_Pell_Vell} 
with Lemma \ref{lemma_rel_Pell_Greens} .
\end{proof}

\begin{corollary}\label{cor_rel_Eell_Gell}
For $z\in\mathbb{H}$ with $z\neq\gamma e_{j}$ for any $\gamma\in\Gamma$, at $s=0$, we have the Laurent expansion
\begin{align*}
\mathcal{E}^{\mathrm{ell}}_{e_j}(z,s)- \frac{2^{s+1}\sqrt{\pi}\,\Gamma(s+\frac{1}{2})}{\Gamma(s)} \,G^{\mathrm{ell}}_{e_j}(z,s)= O(s^2).
\end{align*}
\end{corollary}

\begin{proof}
Starting from Proposition \ref{prop_rel_Eell_Gell}, one can meromorphically continue the
relation \eqref{rel_Eell_Gell} to all $s\in\mathbb{C}$ with $-1<\Re(s)<1$. From this relation, 
the claimed Laurent expansion can easily be proven.
\end{proof}

\begin{remark}
The question arises what kind of applications can be deduced from
the above relation between the elliptic Eisenstein
series and the automorphic Green's function. We leave this for future studies.
\end{remark}

\bibliography{bibliography}
\bibliographystyle{amsalpha} 
%



\end{document}